\documentclass[final,hidelinks,onefignum,onetabnum]{siamart251216}

\usepackage{enumitem}
\usepackage{subfig}
\usepackage{amsmath}
\usepackage{bm}
\usepackage{tikz}
\usetikzlibrary{positioning,shapes.geometric}
\usepackage{appendix}
\usepackage{amsfonts}
\usepackage{multirow}
\usepackage{amssymb}
\usepackage{algorithm,algorithmicx,algpseudocode}
\usepackage{extarrows} 

\usepackage[marginal,hang]{footmisc}

\usepackage{booktabs}
\makeatletter
\g@addto@macro\normalsize{%
	\setlength{\abovedisplayskip}{5.5pt plus 2pt minus 2pt}%
	\setlength{\belowdisplayskip}{5.5pt plus 2pt minus 2pt}%
	\setlength{\abovedisplayshortskip}{2.5pt plus 1pt minus 2pt}%
	\setlength{\belowdisplayshortskip}{2.5pt plus 1pt minus 2pt}%
}
\makeatother
\normalsize

\newsiamremark{remark}{Remark}
\newsiamthm{Def}{Definition}
\newsiamthm{property}{Property}
\newsiamthm{Prop}{Proposition}
\newsiamthm{claim}{Claim}
\newsiamthm{exmp}{Example}
\newsiamthm{assumption}{Assumption}

\ifpdf
\hypersetup{
	pdftitle={},
	pdfauthor={}
}
\fi

\begin{document}

\headers{}{}

\title{{
		Transport Polynomial Exactness for Direct ALE-DG Methods with Nonsmooth mesh motion: A Conservative Coefficient-Evolution Framework}
	\thanks{
		This work was partially supported by National Key R\&D Program of China (Grant No.~2022YFA1004500), Postdoctoral Fellowship Program and China Postdoctoral Science Foundation (Grant No.~BX2026001).}
}

\author{
	Qiqin Cheng\thanks{School of Mathematical Sciences, Xiamen University, Xiamen, Fujian 361005, China (\email{qqcheng@stu.xmu.edu.cn}).}
	\and
	Chaoyi Cai\thanks{Institute of Applied Physics and Computational Mathematics, National Key Laboratory of Computational Physics, 100088, Beijing, China (\email{caichaoyi@stu.xmu.edu.cn}).}
	\and Jianxian Qiu\thanks{Corresponding author. School of Mathematical Sciences and Fujian Provincial Key Laboratory of Mathematical Modeling and High-Performance Scientific Computing, Xiamen University, Xiamen, Fujian 361005, China (\email{jxqiu@xmu.edu.cn}).}
}

\maketitle
\begingroup
\renewcommand\thefootnote{\fnsymbol{footnote}}%
\setcounter{footnote}{0}%
\setlength{\footnotemargin}{1.1em}
\endgroup
\begin{abstract} 
	High-order direct arbitrary Lagrangian--Eulerian discontinuous Galerkin (ALE-DG) methods can suffer severe order reduction under nonsmooth mesh motion despite satisfying the geometric conservation law (GCL). To address this limitation, we propose a new ALE-DG framework applicable to moving triangular and quadrilateral meshes. This framework successfully overcomes the uniform Lipschitz constraint inherent in existing direct ALE-DG methods, retaining high-order accuracy even when the mesh-velocity gradients are unbounded. This capability is driven by transport polynomial exactness (TPE($k$)), which is a structural property stronger than the GCL. TPE($k$) ensures the exact transport of polynomials of degree at most $k$ under arbitrary mesh motion. To incorporate TPE($k$) into a direct ALE-DG method, we introduce a coefficient-evolution (CE) formulation. Reconciling strict local conservation, an indispensable requirement for accurate shock capturing, with the preservation of the TPE(\(k\)) property constitutes the principal technical challenge. We resolve this by evolving the coefficients using the third-order strong-stability-preserving Runge--Kutta method, paired with a conservative correction executed via Simpson's rule integration. Crucially, this correction strictly vanishes for exactly transported quadratic polynomials, thereby ensuring that our third-order scheme rigorously satisfies the TPE($2$) property. Numerical experiments demonstrate that the proposed scheme retains third-order accuracy when the discrete mesh-velocity Lipschitz estimate grows as $\mathcal O(h^{-1})$, whereas a representative existing direct ALE-DG scheme exhibits order reduction under the same mesh motion. Tests with discontinuous solutions further demonstrate robust and high-resolution performance under random and Lagrangian mesh motions.
\end{abstract}


\section{Introduction}
Hyperbolic conservation laws express a basic physical principle: in the absence of sources, quantities such as mass, momentum, and energy change only through fluxes across the boundary of a control volume. They arise, for example, in compressible gas dynamics, acoustics, shallow water flow, and multi-material flow~\cite{anderson2018high,leveque2002finite}, and generally take the form
\begin{equation}
	\label{eq:MainEq}
	\frac{\partial \bm u}{\partial t}
	+\nabla\cdot\bm F(\bm u)=\bm 0,
	\qquad
	\bm x\in\Omega(t)\subset\mathbb R^d,\quad t>0,
\end{equation}
where $\Omega(t)$ is the computational domain, $d$ is the spatial dimension, $\bm u(\bm x,t)\in\mathbb R^m$ is the vector of conserved variables, and $\bm F(\bm u)\in\mathbb R^{m\times d}$ is the flux. Despite its compact form, equation \eqref{eq:MainEq} can generate smooth waves together with sharply localized and moving structures, including shocks and contact discontinuities~\cite{anderson2018high,leveque2002finite}. A fixed mesh can resolve these structures by using sufficiently small cells throughout the domain, but this may spend most degrees of freedom away from the regions where they are needed. Moving-mesh methods instead redistribute the cells as the solution evolves, with the aim of concentrating resolution near dynamically important features~\cite{Budd_Huang_Russell_2009,tang2003adaptive}. Among them, the arbitrary Lagrangian--Eulerian (ALE) method provides a flexible framework by allowing the mesh velocity to be prescribed independently of the material velocity, thereby bridging the purely Eulerian formulation (fixed mesh) and the fully Lagrangian approach (tracking the material)~\cite{Donea2004Chapter1A,HIRT1974227}.

The ALE methodology encompasses two principal computational frameworks: indirect ALE~\cite{galera2010two} and direct ALE~\cite{boscheri2014direct}. The indirect ALE framework commonly consists of three successive stages: Lagrangian evolution~\cite{CHENG20071567}, mesh rezoning~\cite{BRACKBILL1982342}, and conservative remapping~\cite{dukowicz1987accurate}. In contrast, within the direct ALE framework, the governing equations are discretized and evolved directly on continuously moving cells, so that the numerical solution and mesh geometry evolve simultaneously and no separate remapping step is required~\cite{boscheri2014direct,boscheri2017aledg}. The two routes are complementary, but they maintain high-order accuracy in different ways: an indirect ALE method can control transfer errors through a distinct remapping step, whereas a direct ALE method must preserve compatibility between the evolving mesh geometry and the solution degrees of freedom throughout the coupled evolution. This work focuses on the latter setting.

High-order numerical schemes are particularly important for accurately resolving smooth waves and fine-scale structures with low numerical dissipation. As a fundamental component of such schemes, high-order spatial discretization is essential for achieving their designed order of accuracy. Discontinuous Galerkin (DG) methods are particularly well suited for constructing high-order spatial discretizations within the ALE framework because of their compact stencils, conservation, and flexibility on general meshes~\cite{COCKBURN198990,cockburn1989tvb,reed1973triangular}. These features have motivated extensive research on extending DG methods to moving meshes within the direct ALE framework. Representative high-order direct ALE-DG formulations include, but are not limited to, those developed in \cite{acosta2011moving,boscheri2017aledg,fu2019simplex,klingenberg2017arbitrary,lomtev1999aledg,luo2019quasilagrangian,mavriplis2011geometric,schnucke2020entropy}. Several of these studies developed high-order schemes and demonstrated the expected high-order convergence rates through numerical experiments. More importantly, some of them went further by providing rigorous error analyses that establish high-order convergence. For one-dimensional conservation laws, Klingenberg et~al. established $L^2$ stability and a priori error estimates for the semidiscrete ALE-DG method~\cite{klingenberg2017arbitrary}, while Zhou et~al. extended the analysis to fully discrete schemes coupled with explicit Runge--Kutta time discretizations~\cite{zhou2019stability}. Further one-dimensional studies addressed superconvergence and negative-norm error estimates for linear and nonlinear hyperbolic equations, including problems involving $\delta$-singularities~\cite{hong2020delta,tao2019superconvergence,tao2022negative}. Fu et~al. extended the analysis to two-dimensional conservation laws on moving triangular meshes, establishing $L^2$ stability and an a priori error estimate for the corresponding ALE-DG method~\cite{fu2019simplex}.

Collectively, the error analyses reviewed above have laid an important theoretical foundation for high-order direct ALE-DG schemes on moving meshes~\cite{fu2019simplex,hong2020delta,klingenberg2017arbitrary,tao2019superconvergence,tao2022negative,zhou2019stability}. In these analyses, mesh motion enters the variational and error equations through the mesh velocity and its spatial gradient~\cite{fu2019simplex,klingenberg2017arbitrary}. Most of the corresponding high-order error estimates are established under mesh-motion assumptions that are uniform with respect to the mesh size, requiring these quantities to remain uniformly bounded~\cite{fu2019simplex,hong2020delta,tao2019superconvergence,tao2022negative,zhou2019stability}. For the continuous piecewise-affine mesh velocities adopted in these direct ALE formulations~\cite{fu2019simplex,klingenberg2017arbitrary,zhou2019stability}, the gradient bound amounts to a mesh-independent Lipschitz condition. In practical ALE computations, however, such a uniform bound need not hold. For example, a Lagrangian mesh velocity follows the local material velocity and can therefore vary sharply near evolving shocks or other nonsmooth flow structures. Consequently, its discrete Lipschitz estimate may grow as $\mathcal{O}(h^{-1})$ under mesh refinement.
In such cases, the existing estimates no longer guarantee optimal-order accuracy. More importantly, severe order reduction can already be observed in the simplest possible setting, as demonstrated by the following diagnostic test. We consider the stationary problem
\begin{equation}
	\label{eq}
	\partial_t u=0,
\end{equation}
with the initial condition $u_0(x,y)=1+0.2\sin(2\pi(x+y))$. We compare two third-order direct ALE-DG schemes on the same randomly moving triangular meshes. Scheme 1 denotes the direct ALE-DG scheme developed by Fu et~al.~\cite{fu2019simplex}, with nonlinear limiting disabled, whereas Scheme 2 denotes the scheme proposed in this work; both schemes satisfy the geometric conservation law (GCL). The common mesh motion is driven by a nonsmooth mesh velocity whose discrete Lipschitz estimate scales as $\mathcal{O}(h^{-1})$ under mesh refinement.
\begin{table}[htbp]
	\centering 
	\caption{Comparison of Schemes 1 and 2 for the stationary problem at
		$t=0.1$ on $\Omega(0)=[0,1]\times[0,1]$: $L^\infty$ errors, convergence
		orders, and observed discrete mesh-velocity Lipschitz estimates
		$\mathrm{Lip}(\bm\mu^g)$. The structured triangular meshes are generated
		by dividing each cell of a uniform $N_x\times N_y$ Cartesian partition
		into two triangles.}
	\label{tab:stationary-test}
	\scriptsize
	\begin{tabular}{c c c c c c c}
		\hline
		& $N_x \times N_y$
		& $40\times40$ & $80\times80$ & $160\times160$
		& $320\times320$ & $640\times640$ \\
		\hline
		& $L^\infty$ error
		& 4.2466e-05 & 9.9401e-06 & 3.0270e-06
		& 1.5167e-06 & 7.9498e-07 \\
		Scheme 1 & Order
		& -- & 2.09 & 1.72 & 1.00 & 0.93 \\
		& $\mathrm{Lip}(\bm\mu^g)$
		& 434.84 & 893.74 & 1891.91 & 3976.94 & 8270.39 \\
		\hline
		& $L^\infty$ error
		& 5.5310e-05 & 7.5622e-06 & 9.5602e-07
		& 1.2590e-07 & 1.5991e-08 \\
		Scheme 2 & Order
		& -- & 2.87 & 2.98 & 2.92 & 2.98 \\
		& $\mathrm{Lip}(\bm\mu^g)$
		& 434.84 & 893.74 & 1891.91 & 3976.94 & 8270.39 \\
		\hline
	\end{tabular}
\end{table}
As shown in Table~\ref{tab:stationary-test}, Scheme 1 exhibits progressive order reduction under mesh refinement, with its convergence rate eventually decreasing to approximately first order, whereas Scheme 2 consistently retains third-order accuracy under the same mesh motion. This loss of accuracy cannot be attributed to physical advection, flux nonlinearity, mesh regularity degradation, or limiting procedures. Instead, it reveals a structural limitation that can arise in existing direct ALE-DG schemes: the GCL alone does not prevent mesh-motion-induced order reduction. This observation motivates a stronger structural criterion that extends exactness from constants to higher-degree polynomials.

An analogous polynomial-exactness concept appears in Barth's $k$-exact reconstruction, which recovers polynomials of degree at most $k$ from cell data~\cite{barth1993recent}. This notion, however, concerns spatial reconstruction at a fixed time and does not characterize the polynomial exactness of a fully discrete scheme on an evolving mesh. At the fully discrete level, transport polynomial exactness of order $k$ (TPE($k$)) was first introduced in~\cite{cai2026beyond} for rezoning moving-mesh methods, with related evolved-geometric-moment ideas developed for cut-cell methods in~\cite{cai2025gpr}. A moving-mesh method satisfies TPE($k$) if, for the linear advection equation, it exactly transports every polynomial of degree at most $k$ under arbitrary mesh motion. In particular, TPE($0$) reduces to the GCL, whereas TPE($k$) for $k\geq1$ extends this exactness from constant states to degree-$k$ polynomials. Since a smooth solution is locally represented by its degree-$k$ Taylor polynomial up to a higher-order remainder, exact transport of this polynomial content through TPE($k$) helps suppress lower-order errors induced by mesh motion, thereby providing a mechanism for retaining high-order accuracy when a mesh-independent Lipschitz bound on the mesh velocity is unavailable.

Motivated by this mechanism, the main objective of this work is to incorporate TPE($k$) into a direct ALE-DG framework. Achieving this objective is nontrivial. During direct ALE evolution, the mesh geometry, basis functions, and solution degrees of freedom evolve simultaneously, so exact polynomial transport must be maintained throughout this fully coupled process. At the same time, the scheme must satisfy strict local conservation, which is essential for correctly capturing shock propagation. These two requirements are not automatically compatible: an evolution that preserves the target polynomials may fail to maintain conservative cell totals, whereas enforcing conservation may perturb the polynomial modes and destroy TPE($k$). The main challenge is therefore to construct a fully discrete scheme that enforces conservation for general solutions while leaving the target polynomial solutions unchanged. To resolve this compatibility problem, we develop a new ALE-DG framework. The main contributions of this work are summarized as follows:
\begin{itemize}
	\item \textbf{The first use of TPE($k$) to assess direct ALE methods.}
	To the best of our knowledge, TPE($k$) has not previously been applied in the direct ALE setting. Whereas the GCL only characterizes exact preservation of constant states, TPE($k$) examines the exact transport of polynomials of degree up to $k$ and thereby provides a higher-order structural criterion for diagnosing accuracy loss induced by mesh motion.
	\item \textbf{The direct coefficient-evolution (CE) formulation satisfies the GCL and TPE($k$).}
	We develop a coefficient-evolution ALE-DG (CE-ALE-DG) formulation by deriving evolution equations directly for the DG polynomial coefficients, explicitly accounting for the time dependence of the moving-cell geometry and basis functions. The resulting formulation is proved to satisfy GCL and TPE($k$) without requiring a mesh-independent Lipschitz bound on the mesh velocity.
	\item \textbf{A conservative third-order fully discrete CE scheme is constructed.}
	Within each time step, the local coefficients are evolved using the third-order strong stability preserving Runge--Kutta (SSPRK3) method, while an auxiliary cell total is updated by applying Simpson's rule to the ALE boundary flux using states reconstructed from the SSPRK3 stages. Then, a correction is applied to the provisional coefficient vector so that its cell integral agrees exactly with the conservatively updated total. This coupled construction incorporates local conservation into direct coefficient evolution.
	\item \textbf{Conservation is enforced without altering the GCL or TPE($2$).} The correction vanishes for constant states and exactly transported quadratic polynomials, and therefore leaves the fully discrete GCL and TPE($2$) properties unchanged. For bilinearly mapped quadrilateral cells, a separate quadrature-exactness analysis accounts for the spatially varying Jacobian and establishes the fully discrete GCL and TPE($2$) properties on quadrilateral meshes.
	\item \textbf{Accuracy and robustness under nonsmooth mesh motion.} The numerical results confirm the scheme's conservation, GCL, and TPE($2$) properties and show that the scheme retains third-order accuracy under nonsmooth mesh motion without a uniform Lipschitz bound. Numerical tests involving shock waves further demonstrate stable performance under random and Lagrangian mesh motions. 
\end{itemize}

The remainder of this paper is organized as follows. Section~\ref{sec:Preliminaries} introduces the notation and TPE($k$) property. Section~\ref{sec:CE-ALE-DG} develops the CE-ALE-DG formulation, establishes its GCL and TPE($k$) structures, and constructs a third-order conservative scheme on moving triangular and quadrilateral meshes. Section~\ref{sec:Num_examples} presents numerical tests of its structural properties, accuracy, and robustness. Section~\ref{sec:Conclusion} concludes the paper.
\section{Preliminaries}\label{sec:Preliminaries}
In this section, we introduce the notation, recall the TPE($k$) property, and explain its relevance to accuracy under mesh motion.
\subsection{Notation} For simplicity, we consider the two-dimensional scalar form of \eqref{eq:MainEq} in Sections~\ref{sec:Preliminaries} and \ref{sec:CE-ALE-DG}, where $u(\bm x,t)$ is the scalar conserved variable and $\bm F(u)\in\mathbb R^2$ is the corresponding flux function. 

Let $\mathcal T_h(t)$ be a shape-regular moving mesh of $\Omega(t)$, consisting of nonoverlapping cells $I(t)$. On each cell $I(t)$, let $\{\phi_{I,\ell}(\bm x,t)\}_{\ell=1}^N$ be a basis of
$\mathbb P^k(I(t))$, so that
\[
\mathbb V_I^k(t):=\mathbb P^k(I(t))
=\operatorname{span}\{\phi_{I,1}(\bm x,t),\ldots,
\phi_{I,N}(\bm x,t)\}.
\]
For notational simplicity, we write $\bm\phi=(\phi_1,\ldots,\phi_N)^T$ and $\bm c=(c_1,\ldots,c_N)^T$. The DG solution is then represented as $u_h=\bm\phi^T\bm c$.
\subsection{TPE($k$) property and motivation}
To characterize polynomial transport beyond constant-state preservation under the GCL, we recall the TPE($k$) property introduced in~\cite{cai2026beyond}. Consider the linear advection equation
\begin{equation}\label{eq:linearAdvEquation}
	u_t+\bm a\cdot\nabla u=0,
\end{equation}
where $\bm a = (a_1,a_2)^T$ is a constant vector. If the initial condition is a polynomial $q(\bm x)\in\mathbb P^k$, then the exact solution is $u(\bm x,t)=q(\bm x-\bm a t)$, which remains in $\mathbb P^k$.

Let $\mathcal S^{\Delta t}$, $\mathcal P^n_h$, and $\Phi_h^{\Delta t}$ denote the exact solution operator, the DG projection operator at $t_n$, and the update operator of the linear ALE-DG scheme, respectively. For each moving cell $I(t)$, the local one-step consistency error is defined by $\mathcal D_{h,I}^n u := \left. \left[ \Phi_h^{\Delta t}\bigl(\mathcal P_h^n u(t_n)\bigr) -\mathcal P_h^{n+1}\bigl(\mathcal S^{\Delta t}u(t_n)\bigr) \right] \right|_{I(t_{n+1})}$.
\begin{Def}[Transport polynomial exactness]\label{def:TPE}
	For a nonnegative integer~\(k\), a numerical scheme is said to satisfy the TPE($k$) property if \(\mathcal D_{h,I}^n q=0\) for any moving cell $I(t)$, any $q\in\mathbb P^k$, and every time step $t_n$.
\end{Def}
Since $\mathbb P^0$ consists of constant functions, TPE($0$) coincides with the GCL, whereas TPE($k$) for $k\geq1$ extends this exactness to degree-$k$ polynomials. For a smooth solution, let $P_{k,I}$ and $R_{k+1,I}$ be its degree-$k$ Taylor polynomial and remainder on $I(t_n)$, respectively. If the scheme satisfies TPE($k$), linearity gives
\[
\begin{aligned}
	u&=P_{k,I}+R_{k+1,I},
	\qquad R_{k+1,I}=\mathcal{O}(h_I^{k+1}),\\
	\mathcal D_{h,I}^n u
	&=\mathcal D_{h,I}^nP_{k,I}
	+\mathcal D_{h,I}^nR_{k+1,I}
	=\mathcal D_{h,I}^nR_{k+1,I},
\end{aligned}
\]
where  $h_I$ is the diameter of $I(t_n)$. Thus, only the Taylor remainder contributes to the one-step error. Without TPE($k$), the polynomial part may introduce a lower-order error. This observation motivates the CE-ALE-DG construction developed in the next section.
\section{CE-ALE-DG framework}\label{sec:CE-ALE-DG}  
In this section, we develop a conservative CE-ALE-DG framework by evolving the local polynomial coefficients directly and introducing a compatible correction to enforce conservation. We then construct a third-order fully discrete scheme, establish its conservation, GCL, and TPE($2$) properties, and incorporate a discontinuity treatment.
\subsection{CE formulation} Applying the Reynolds transport theorem to the ALE-DG weak form \cite{fu2019simplex} yields \begin{equation}\label{dUtotdt}
	\frac{d}{dt}\bm U_h^{\mathrm{tot}}
	=
	-\int_{\partial I(t)}
	\bigl(\bm G(u_h,\bm\mu^g)\cdot\bm n\bigr)\bm\phi dS
	+\int_{I(t)}
	\nabla\bm\phi\bm F(u_h)d\bm x
	+\int_{I(t)}
	u_h\bm\phi_td\bm x,
\end{equation}
where $\bm U_h^{\mathrm{tot}}(t):=\int_{I(t)}\bm\phi u_hd\bm x$, $\bm\mu^g$ is the mesh velocity, and $\bm G(u_h,\bm\mu^g):=\bm F(u_h)-u_h\bm\mu^g$ is the ALE flux. On each cell $I(t)$, $\bm U_h^{\mathrm{tot}}(t)$ and $\bm c(t)$ satisfy
\[
\bm U_h^{\mathrm{tot}}(t)=M(t)\bm c(t),
\text{ where }
M(t)=\int_{I(t)}\bm\phi\bm\phi^Td\bm x.
\]
Differentiating \(\bm U_h^{\mathrm{tot}}(t)=M(t)\bm c(t)\) with respect to time and evaluating \(dM/dt\) by the Reynolds transport theorem give
\[
\frac{d\bm c}{dt} =M^{-1}(t)\left( \frac{d\bm U_h^{\mathrm{tot}}}{dt}-\frac{dM}{dt}\bm c \right),
\qquad
\frac{dM}{dt} = \int_{I(t)}(\bm\phi\bm\phi^T)_td\bm x
+\int_{\partial I(t)}(\bm\mu^g\cdot\bm n)
\bm\phi\bm\phi^TdS.
\]
Substituting \eqref{dUtotdt} and $dM/dt$ into $d\bm c/dt$ yields
\begin{equation}
	\begin{aligned}
		\frac{d}{dt}\bm{c}(t) &= M^{-1}(t)\left(-\int_{I(t)}\left(\bm{\phi}^T_t\bm{c}(t)\right)\bm{\phi}d\bm x - \int_{\partial I(t)}\bm{\mu}^{g}\cdot \bm{n}\bm{\phi}\bm{\phi}^T\bm{c}(t)dS\right.\\
		&\left.-\int_{\partial I(t)}\bm{G}(u_h,\bm{\mu}^g)\cdot\bm{n}\bm{\phi}dS
		+ \int_{I(t)}\nabla \bm{\phi}\bm{F}(u_h) d\bm x\right).
	\end{aligned}
	\label{dctdt3}
\end{equation}
We write \eqref{dctdt3} as $d\bm c/dt=\mathcal R(\bm c,t)$. For a static mesh, $\bm\mu^g=0$ and $\bm\phi_t=0$, so \eqref{dctdt3} reduces to the standard DG formulation~\cite{cockburn1989tvb}.

The distinctive feature of \eqref{dctdt3} is that the polynomial coefficients are evolved directly. As shown in the next subsection, this feature leads to the cancellation of the mesh-motion contributions for exactly transported polynomials and is therefore essential to the GCL and TPE($k$) properties.
\subsection{Structural properties of the CE formulation}
To construct a fully discrete scheme that preserves the GCL and TPE($k$), we first establish these properties for \eqref{dctdt3}. This separates the intrinsic structure of the CE formulation from the effects of numerical discretization.
\subsubsection{GCL structure}
\begin{proposition}
	If $u_h(\cdot,t_n)=\bm\phi^T(\cdot,t_n)\bm c(t_n)\equiv C$, then the solution of \eqref{dctdt3} satisfies $u_h(\cdot,t)\equiv C$ for $t\in[t_n,t_{n+1}]$.
\end{proposition}

\begin{proof}
	Let $\bm\alpha(t)$ be defined by
	\begin{equation}\label{eq:phiAlphaEqOne}
		\bm\phi^T(x,y,t)\bm\alpha(t)=1,
	\end{equation}
	and set $\widehat{\bm c}(t)=C\bm\alpha(t)$, which represents the constant state $C$ in the basis $\bm\phi$.
	
	We verify that \(\widehat{\bm c}(t)\) solves \eqref{dctdt3}. Evaluating the right-hand side of \eqref{dctdt3} at \(\widehat{\bm c}(t)\), and using \(\bm\phi^T\widehat{\bm c}=C\) and \(\bm G(C,\bm\mu^g)=\bm F(C)-C\bm\mu^g\), gives
	\[
	\begin{aligned}
		\mathcal R(\widehat{\bm c},t)
		={}&M^{-1}(t)\Bigg[
		-C\int_{I(t)}(\bm\phi_t^T\bm\alpha)\bm\phi d\bm x-C\int_{\partial I(t)}(\bm\mu^g\cdot\bm n)\bm\phi dS\\
		&-\int_{\partial I(t)}
		(\bm F(C)-C\bm\mu^g)\cdot\bm n\bm\phi dS+\int_{I(t)}\nabla\bm\phi \bm F(C) d\bm x\Bigg]\\
		=&-CM^{-1}(t)\int_{I(t)}
		(\bm\phi_t^T\bm\alpha)\bm\phi d\bm x\\
		&+M^{-1}(t)\Bigg[
		-\int_{\partial I(t)}(\bm F(C)\cdot\bm n)\bm\phi dS
		+\int_{I(t)}\nabla\bm\phi \bm F(C) d\bm x\Bigg].
	\end{aligned}
	\]
	Because $\bm F(C)$ is constant, the divergence theorem gives
	\[
	\int_{I(t)}\nabla\bm\phi\bm F(C)d\bm x
	=\int_{\partial I(t)}(\bm F(C)\cdot\bm n)\bm\phi dS.
	\]
	Therefore,
	\[
	\mathcal R(\widehat{\bm c},t)
	=-CM^{-1}(t)\int_{I(t)}
	(\bm\phi_t^T\bm\alpha)\bm\phi d\bm x.
	\]
	
	Differentiating \eqref{eq:phiAlphaEqOne} gives
	$-\bm\phi_t^T\bm\alpha=\bm\phi^T\bm\alpha'$, and therefore
	\[
	\mathcal R(\widehat{\bm c},t)
	=CM^{-1}(t)M(t)\bm\alpha'(t)
	=\frac{d\widehat{\bm c}}{dt}.
	\]
	Since $\widehat{\bm c}(t_n)=\bm c(t_n)$, uniqueness of the coefficient ordinary differential equation (ODE) yields $\bm c(t)=C\bm\alpha(t)$ and hence $u_h(\cdot,t)\equiv C$.
\end{proof}

\subsubsection{\texorpdfstring{TPE$(k)$}{TPE(k)} structure}
In this part, we consider the linear advection equation \eqref{eq:linearAdvEquation}.
Let $\bm m(x,y)$ be a fixed vector whose components form a basis of $\mathbb P^k$. Since $\mathbb P^k\subseteq\mathbb V_I^k(t)$, there exists a unique matrix $T(t)$ such that $\bm\phi^T(x,y,t)T(t)=\bm m^T(x,y).$ Moreover, every $u(\cdot,t)\in\mathbb P^k$ admits a unique representation $u(x,y,t)=\bm m^T(x,y)\bm d(t)$.
\begin{lemma}\label{lem:NablaMAndA}
	There exists a constant matrix $A$ such that
	$-\bm a\cdot\nabla\bm m^T=\bm m^TA$. Moreover, a polynomial solution of \eqref{eq:linearAdvEquation} satisfies $\bm d'(t)=A\bm d(t)$.
\end{lemma}
\begin{proof}
	The operator $-\bm a\cdot\nabla$ maps $\mathbb P^k$ into $\mathbb P^{k-1}\subset\mathbb P^k$, which defines $A$. Substitution of $u=\bm m^T\bm d$ into \eqref{eq:linearAdvEquation} and linear independence of the components of $\bm m$ give $\bm d'(t)=A\bm d(t)$.
\end{proof}
\begin{proposition}
	Let $q\in\mathbb P^k$. If
	$u_h(x,y,t_n)=q(x-a_1t_n,y-a_2t_n)$, then the solution of \eqref{dctdt3} satisfies
	$u_h(x,y,t)=q(x-a_1t,y-a_2t)$ for $t\in[t_n,t_{n+1}]$.
\end{proposition}

\begin{proof}
	Set $\widehat{\bm c}(t)=T(t)\bm d(t)$, so that
	$\bm\phi^T\widehat{\bm c}=\bm m^T\bm d=u$. For \eqref{eq:linearAdvEquation},
	$\bm F(u)=\bm a u$ and $\bm G(u,\bm\mu^g)=u(\bm a-\bm\mu^g)$.
	Substitution into \eqref{dctdt3} gives
	\[
	\begin{aligned}
		\mathcal R(\widehat{\bm c},t)
		={}&M^{-1}(t)\Bigg[
		-\int_{I(t)}(\bm\phi_t^TT\bm d)\bm\phi d\bm x-\int_{\partial I(t)}u(\bm\mu^g\cdot\bm n)\bm\phi dS\\
		&-\int_{\partial I(t)}u(\bm a-\bm\mu^g)\cdot\bm n \bm\phi dS+\int_{I(t)}\nabla\bm\phi\bm a u d\bm x\Bigg]\\
		={}&M^{-1}(t)\Bigg[
		-\int_{I(t)}(\bm\phi_t^TT\bm d)\bm\phi d\bm x-\int_{\partial I(t)}u(\bm a\cdot\bm n)\bm\phi dS
		+\int_{I(t)}\nabla\bm\phi\bm a u d\bm x\Bigg].
	\end{aligned}
	\]
	The mesh-velocity terms therefore cancel. Integration by parts gives
	\[
	\int_{I(t)}\nabla\bm\phi\,\bm a u d\bm x
	-\int_{\partial I(t)}u(\bm a\cdot\bm n)\bm\phi dS
	=-\int_{I(t)}\bm\phi(\bm a\cdot\nabla u)d\bm x,
	\]
	and hence
	\[
	\mathcal R(\widehat{\bm c},t)
	=M^{-1}(t)\left[
	-\int_{I(t)}(\bm\phi_t^TT\bm d)\bm\phi d\bm x
	-\int_{I(t)}\bm\phi\,\bm a\cdot\nabla u d\bm x
	\right].
	\]
	Since $\bm\phi^TT=\bm m^T$ and
	$-\bm a\cdot\nabla\bm m^T=\bm m^TA$,
	\[
	-\bm\phi_t^TT=\bm\phi^TT',
	\qquad
	-\bm a\cdot\nabla u=\bm\phi^TTA\bm d.
	\]
	Consequently,
	\[
	\mathcal R(\widehat{\bm c},t)
	=T'(t)\bm d(t)+T(t)A\bm d(t)
	=\frac{d}{dt}\bigl(T(t)\bm d(t)\bigr).
	\]
	Since the candidate and actual coefficient vectors have the same initial data, ODE uniqueness gives $\bm c(t)=T(t)\bm d(t)$ and therefore $u_h(\cdot,t)=u(\cdot,t)$.
\end{proof}

The preceding results show that \eqref{dctdt3} preserves constant states and exactly transports every polynomial in $\mathbb P^k$. The following subsection develops a fully discrete scheme that preserves both properties.

\subsection{Third-order conservative scheme construction}
Based on the GCL and TPE($k$) structures established above, we now construct a third-order conservative fully discrete scheme applicable to both moving triangular and quadrilateral meshes and prove its GCL and TPE($2$) properties. Throughout each time interval $[t_n,t_{n+1}]$, every mesh vertex is assumed to move along a trajectory that is affine in time. The resulting cells are assumed to remain nondegenerate and shape-regular throughout the time step.
\subsubsection{Choice of basis functions}\label{FrozenBasis}
To simplify the coefficient evolution on a moving cell, we freeze a local physical Taylor basis within each time step, which eliminates $\phi_t$ in \eqref{dctdt3} while retaining the same polynomial approximation space. For $t \in [t_n,t_{n+1}]$, define
\begin{equation*}
	X_I(t_n)=\frac{x-x_I(t_n)}{\sqrt{|I(t_n)|}},\qquad Y_I(t_n)=\frac{y-y_I(t_n)}{\sqrt{|I(t_n)|}},
\end{equation*}
where $\bm x_I(t_n)=(x_I(t_n),y_I(t_n))$ is the cell centroid. For the DG space $\mathbb V_I^2(t)$, the local basis is ordered as
\[
\bm\phi(t_n)
=
\left(
1,\,
X_I(t_n),\,
Y_I(t_n),\,
X_I^2(t_n),\,
X_I(t_n)Y_I(t_n),\,
Y_I^2(t_n)
\right)^T.
\]
\textbf{Basis transformation.} The coefficient vector $\widetilde{\bm c}(t_{n+1})$ is then transformed from \(\bm\phi(t_n)\) to \(\bm\phi(t_{n+1})\) at \(t_{n+1}\). Define the cell-scale ratio and the normalized centroid shifts by
\begin{equation*}
	s_I = \sqrt{\frac{|I(t_{n+1})|}{|I(t_{n})|}}, \quad 
	d_x = \frac{x_I(t_{n+1}) - x_I(t_{n})}{\sqrt{|I(t_{n})|}}, \quad 
	d_y = \frac{y_I(t_{n+1}) - y_I(t_{n})}{\sqrt{|I(t_{n})|}}.
\end{equation*}

Since $\bm\phi(t_n)$ and $\bm\phi(t_{n+1})$ span the same DG space,
$\bm\phi^T(t_n)=\bm\phi^T(t_{n+1})P^T$, where
\begin{equation*}
	P^T = 
	\begin{pmatrix}
		1 & d_x & d_y & d_x^2 & d_x d_y & d_y^2 \\
		0 & s_I   & 0   & 2s_I d_x & s_I d_y & 0 \\
		0 & 0   & s_I   & 0     & s_I d_x & 2s_I d_y \\
		0 & 0   & 0   & s_I^2   & 0     & 0 \\
		0 & 0   & 0   & 0     & s_I^2   & 0 \\
		0 & 0   & 0   & 0     & 0     & s_I^2
	\end{pmatrix}.
\end{equation*}

Thus, the provisional coefficients in the frozen basis are transformed to the new basis by
\begin{equation}\label{eq:basis_update}
	\bm c(t_{n+1})
	=
	P^T\widetilde{\bm c}(t_{n+1}).
\end{equation}
This transformation leaves the represented polynomial \(u_h\) unchanged.
\subsubsection{Scheme discretization}
Before presenting the discrete scheme, we introduce the boundary quadrature operator \(\mathcal Q_{\partial I(t)}\) and the volume quadrature operator \(\mathcal Q_{I(t)}\). Let $\{\omega_q^e\}_{q=1}^{G_e}$ and $\{\omega_q^I\}_{q=1}^{G_I}$ denote the normalized edge quadrature weights and the reference-element volume quadrature weights, respectively. For vector-valued $\bm f(u)$,
\begin{equation*}
	\begin{aligned}
		\int_{\partial I(t)}\bm f(u)\,dS
		&\approx
		\mathcal Q_{\partial I(t)}(\bm f(u))
		=
		\sum_l|\partial I_l(t)|
		\sum_{q=1}^{G_e}
		\omega_q^e\bm f(u_{l,q}),\\
		\int_{I(t)}\bm f(u)d\bm x
		&\approx
		\mathcal Q_{I(t)}(\bm f(u))
		=
		\sum_{q=1}^{G_I}
		\omega_q^I|J|_q\bm f(u_q),
	\end{aligned}
\end{equation*}
where $G_e$ and $G_I$ are the numbers of edge and cell quadrature points, and $|J|_q$ is the Jacobian determinant of the element mapping. We use four-point Gauss--Lobatto edge quadrature, six-point triangular quadrature~\cite{dunavant1985high}, and $4\times4$ tensor-product Gauss--Lobatto quadrature on quadrilaterals. These rules exactly integrate the polynomial terms considered in Lemma \ref{lem:QuadExact}.

Then, the spatial discretization of \eqref{dctdt3} is given by
\begin{equation}
	\label{semidis}
	\begin{aligned}
		\frac{d}{dt}\bm{c}(t) &= M^{-1}(t)\left(-\mathcal{Q}_{\partial I(t)}\left(\bm{\mu}^{g}\cdot \bm{n}\bm{\phi}\bm{\phi}^T\bm{c}(t)\right) \right.\\
		&\left.-\mathcal{Q}_{\partial I(t)}\left(\widehat{g}(u^{-}_{h},u^{+}_{h};\bm{\mu}^g,\bm{n})\bm{\phi}\right) + \mathcal{Q}_{I(t)}\left(\nabla \bm{\phi}\bm{F}(u_h) \right)\right).
	\end{aligned}
\end{equation}
The right-hand side of \eqref{semidis} is denoted by $\mathcal R_{I,0}(M(t),\bm c(t))$.

The ALE local Lax--Friedrichs numerical flux~\cite{fu2019simplex} is defined by
\begin{equation*}
	\widehat{g}_{lk}(u^{-}_{h},u^{+}_{h};\bm{\mu}^g,\bm{n}) = \frac{1}{2}\left(\bm{G}(u^{-}_{h,lk},\bm{\mu}_{lk}^g)+\bm{G}(u^{+}_{h,lk},\bm{\mu}_{lk}^g)\right)\cdot\bm{n}_{l}-\frac{1}{2}\alpha_{l}\left(u^{+}_{h,lk}-u^{-}_{h,lk}\right),
\end{equation*}
where $\alpha_l$, the maximum spectral radius of the ALE normal Jacobian on edge $l$, is defined by
\begin{equation*}
	\alpha_l=\max_{u\in\{u_h^-,u_h^+\}}
	r\left(\bm{F}_u(u)\cdot\bm n_{l}-(\bm\mu^g\cdot\bm n_l)\mathbf{I}\right),
\end{equation*}
and $r(\cdot)$ denotes the spectral radius. Here $\mathbf{I}$ is the identity matrix.

We next discretize the semidiscrete system \eqref{semidis} in time using the SSPRK3 method~\cite{shu1988efficient}:
\begin{equation}
	\label{disScheme}
	\begin{aligned}
		\bm{c}(t^{(0)}) &= \bm{c}(t_{n}),\\
		\bm{c}(t^{(1)}) &= \bm{c}(t^{(0)}) + \Delta t\mathcal{R}_{I,0}\left(M(t^{(0)}),\bm{c}(t^{(0)})\right),\\
		\bm{c}(t^{(2)}) &= \frac{3}{4}\bm{c}(t^{(0)}) +\frac{1}{4}\left[\bm{c}(t^{(1)}) + \Delta t\mathcal{R}_{I,0}\left(M(t^{(1)}),\bm{c}(t^{(1)})\right)\right],\\
		\bm{c}(t^{(3)}) &= \frac{1}{3}\bm{c}(t^{(0)}) +\frac{2}{3}\left[\bm{c}(t^{(2)}) + \Delta t\mathcal{R}_{I,0}\left(M(t^{(2)}),\bm{c}(t^{(2)})\right)\right],\\
		\widetilde{\bm c}(t_{n+1})&=\bm c(t^{(3)}).
	\end{aligned}
\end{equation}
\begin{remark}
	The SSPRK3 geometry stage times are
	\[
	t^{(0)}=t_n,\qquad
	t^{(1)}=t_{n+1},\qquad
	t^{(2)}=t_{n+\frac12},\qquad
	t^{(3)}=t_{n+1}.
	\]
	Geometry-dependent quantities are evaluated at their corresponding stage times. The vectors $\bm c(t^{(1)})$ and $\bm c(t^{(2)})$ are SSPRK3 stage values, whereas $\bm c(t^{(3)})$ is the final approximation at $t_{n+1}$.
\end{remark}

A key issue with \eqref{disScheme} is that conservation is not automatically guaranteed. Therefore, a conservative correction procedure is designed to restore conservation.
\subsubsection{Conservative correction}
The correction consists of two steps: evolving an auxiliary conservative cell total and adjusting the provisional coefficient vector so that its cell integral matches this total. 

First, the auxiliary cell total is initialized by
$\mathcal U_I^\star(t_n)=\int_{I(t_n)}u_hd\bm x$. The auxiliary total is then evolved according to the conservation law on the moving cell $I(t)$
\begin{equation}
	\label{eq:SDConservedODE}
	\frac{d}{dt}\mathcal U^{\star}_I(t) = -\mathcal{Q}_{\partial I(t)} \left( \widehat g(u_h^-,u_h^+;\bm{\mu}^g,\bm{n}) \right):=\mathcal{R}_{I,1}(\bm{c}(t),t).
\end{equation}
A Simpson-type approximation over $[t_n,t_{n+1}]$ gives
\begin{equation}\label{eq:FDConservedODE}
	\begin{aligned}
		\mathcal U^{\star}_I(t_{n+1}) &= \mathcal U^{\star}_I(t_n)+
		\frac{\Delta t}{6}
		\left[
		\mathcal R_{I,1}(\bm{c}^{E}(t^{(0)}),t_{n})\right.\\
		&\left.+
		4\mathcal R_{I,1}(\bm{c}^{E}(t^{(2)}),t_{n+\frac{1}{2}})
		+
		\mathcal R_{I,1}(\bm{c}^{E}(t^{(1)}),t_{n+1})
		\right],
	\end{aligned}
\end{equation}
where
\begin{equation*}
	\begin{aligned}
		\bm{c}^{E}(t^{(0)})&=\bm{c}(t^{(0)}),\bm{c}^{E}(t^{(1)})=\bm{c}(t^{(3)}),\\
		\bm{c}^{E}(t^{(2)})&=\frac14\bm{c}(t^{(0)})+\frac14\bm{c}(t^{(1)})+\frac12 \bm{c}(t^{(2)}).
	\end{aligned}
\end{equation*}
$\bm{c}^{E}(t^{(2)})$ is an approximation to the coefficient
vector at $t_{n+\frac12}$, while $\bm{c}^{E}(t^{(0)})$ and $\bm{c}^{E}(t^{(1)})$ correspond to $t_n$ and $t_{n+1}$, respectively.

\textbf{Correction.}
After the SSPRK3 update and the basis transformation, the provisional
cell total at $t_{n+1}$ is computed as
\[
\widetilde{\mathcal U}_I(t_{n+1})
=
\int_{I(t_{n+1})}
\bm\phi^T(t_{n+1})\bm c(t_{n+1})d\bm x
=
\bm p_I^T(t_{n+1})\bm c(t_{n+1}),
\]
where $\bm{p}_I(t):=\int_{I(t)}\bm{\phi}(t)d\bm x=\mathcal{Q}_{I(t)}(\bm{\phi}(t)).$

The corrected coefficient vector is defined by
\begin{equation}
	\label{coefConservedUpdate}
	\bm{c}_{Cor}(t_{n+1}) := \bm{c}(t_{n+1})+\lambda_I\bm{\alpha}(t_{n+1}),
\end{equation}
where
\begin{equation*}
	\lambda_I =
	\frac
	{
		\mathcal U_{I}^{\star}(t_{n+1})-\widetilde{\mathcal U}_I(t_{n+1})
	}
	{
		\bm{p}^T_I(t_{n+1})\bm{\alpha}(t_{n+1})
	}.
\end{equation*}
Here $\bm\alpha(t_{n+1})$ is the vector that satisfies $\bm\phi^T(t_{n+1})\bm\alpha(t_{n+1})=1$.
Finally, $\bm c_{Cor}(t_{n+1})$ is taken as the initial coefficient vector for the next time step.
\subsection{Properties of the fully discrete scheme}
The preceding subsection completes the construction of the CE-ALE-DG scheme. In this subsection, we analyze its main discrete properties,
including the conservation, GCL, and TPE($2$) properties.
\subsubsection{Conservation property}
\begin{lemma}\label{lem:correctionTotal}
	If $\bm{c}_{Cor}(t_{n+1})$ is defined by \eqref{coefConservedUpdate}, then 
	\begin{equation*}
		\int_{I(t_{n+1})} u^{Cor}_{h}(t_{n+1}) d\bm x = \mathcal U_I^\star(t_{n+1}),
	\end{equation*}
	where $u^{Cor}_{h}(t_{n+1}):=\bm{\phi}^T\bm{c}_{Cor}(t_{n+1}).$
\end{lemma}
\begin{proof}
	Since $\bm p_I^T\bm\alpha=\int_{I(t_{n+1})}1\,d\bm x=|I(t_{n+1})|>0$, the correction parameter \(\lambda_I\) is well defined. Therefore,
	\[
	\int_{I(t_{n+1})}u_h^{Cor}d\bm x
	=\bm p_I^T\bm c_{Cor}(t_{n+1})
	=\widetilde{\mathcal U}_I(t_{n+1})
	+\lambda_I\bm p_I^T\bm\alpha(t_{n+1})
	=\mathcal U_I^\star(t_{n+1}).
	\]
\end{proof}
\begin{Prop}
	The CE-ALE-DG scheme is locally conservative. Moreover, since the numerical ALE flux is single-valued on each interior edge, the scheme is globally conservative under periodic boundary conditions.
\end{Prop}
\begin{proof}
	First, the initialization of the auxiliary total at $t_n$ gives
	\[
	\mathcal U_I^\star(t_n)
	=
	\int_{I(t_n)}u_h(t_n)d\bm x
	=
	\bm p_I^T(t_n)\bm c(t_n).
	\]
	After the update \eqref{eq:FDConservedODE},
	Lemma~\ref{lem:correctionTotal} gives
	\[
	\mathcal U_I^\star(t_{n+1})
	=
	\int_{I(t_{n+1})}u_h^{Cor}(t_{n+1})d\bm x
	=
	\bm p_I^T(t_{n+1})\bm c_{Cor}(t_{n+1}).
	\]
	
	Next, substituting $\mathcal U_I^\star(t_n)$ and $\mathcal U_I^\star(t_{n+1})$ into
	\eqref{eq:FDConservedODE} yields
	\begin{equation}\label{eq:localCon}
		\begin{aligned}
			&\bm p_I^T(t_{n+1})\bm c_{Cor}(t_{n+1})
			-\bm p_I^T(t_n)\bm c(t_n)\\
			&=\frac{\Delta t}{6}\Big[
			\mathcal R_{I,1}(\bm c^E(t^{(0)}),t_n)
			+4\mathcal R_{I,1}(\bm c^E(t^{(2)}),t_{n+\frac12})
			+\mathcal R_{I,1}(\bm c^E(t^{(1)}),t_{n+1})
			\Big].
		\end{aligned}
	\end{equation}
	Since
	\[
	\mathcal R_{I,1}(\bm c,t)
	=
	-\sum_{e\subset\partial I(t)}
	\mathcal Q_{e(t)}
	\bigl(\widehat g_{I,e}(\bm c,t)\bigr),
	\]
	\eqref{eq:localCon} expresses the change of the cell total entirely
	through boundary fluxes and hence proves local conservation.
	
	Moreover, for an interior edge $e=I^-\cap I^+$,
	$\bm n^+=-\bm n^-$, and the single-valued numerical ALE flux satisfies
	\[
	\widehat g_{I^-,e}(\bm c,t)
	+
	\widehat g_{I^+,e}(\bm c,t)
	=0.
	\]
	Consequently,
	\[
	\sum_I\mathcal R_{I,1}(\bm c,t)
	=
	-\sum_{e\subset\partial\Omega(t)}
	\mathcal Q_{e(t)}
	\bigl(\widehat g_e(\bm c,t)\bigr).
	\]
	
	Finally, summing \eqref{eq:localCon} over all cells gives
	\[
	\begin{aligned}
		&\sum_I\bm p_I^T(t_{n+1})\bm c_{Cor}(t_{n+1})
		-\sum_I\bm p_I^T(t_n)\bm c(t_n)\\
		&=-\frac{\Delta t}{6}\sum_{e\subset\partial\Omega}
		\Big[
		\mathcal Q_{e(t_n)}(\widehat g_e^{(0)})
		+4\mathcal Q_{e(t_{n+\frac12})}(\widehat g_e^{(2)})
		+\mathcal Q_{e(t_{n+1})}(\widehat g_e^{(1)})
		\Big].
	\end{aligned}
	\]
	Thus, the global total changes only through the physical boundary.
	For periodic boundaries, these contributions cancel pairwise, and
	\[
	\sum_I\bm p_I^T(t_{n+1})\bm c_{Cor}(t_{n+1})
	=
	\sum_I\bm p_I^T(t_n)\bm c(t_n).
	\]
\end{proof}
Consequently, the auxiliary-total update \eqref{eq:FDConservedODE} together with the coefficient correction \eqref{coefConservedUpdate} restores local conservation, while the single-valued ALE flux ensures global conservation through pairwise cancellation on interior edges.

Before proving the fully discrete GCL and TPE($2$) properties, we establish the following quadrature exactness result.
\begin{lemma}\label{lem:QuadExact}
	Let $v_h\in\mathbb P^2(I(t))$ and let
	$\bm q\in\mathbb R^2$ be constant. Then
	\[
	\begin{aligned}
		\mathcal Q_{I(t)}(\nabla\bm\phi\bm qv_h)
		&=\int_{I(t)}\nabla\bm\phi\bm qv_hd\bm x,\\
		\mathcal Q_{\partial I(t)}
		\bigl(v_h(\bm q\cdot\bm n)\bm\phi\bigr)
		&=\int_{\partial I(t)}v_h(\bm q\cdot\bm n)\bm\phi dS,
	\end{aligned}
	\]
	and consequently
	\[
	\begin{aligned}
		&\mathcal Q_{I(t)}(\nabla\bm\phi\bm qv_h)
		-\mathcal Q_{\partial I(t)}
		\bigl(v_h(\bm q\cdot\bm n)\bm\phi\bigr)=-\int_{I(t)}\bm\phi(\bm q\cdot\nabla v_h)d\bm x.
	\end{aligned}
	\]
\end{lemma}

\begin{proof}
	Let $\widehat I$ and $\widehat e$ denote the reference cell and edge, respectively, with mappings $\Psi_I:\widehat I\to I(t)$ and $\Psi_e:\widehat e\to e(t)$.
	At each fixed time,
	\[
	\begin{gathered}
		\begin{array}{c|c|c}
			& \text{degree of } \Psi_I & \text{degree of } |J_I| \\ \hline
			\text{triangle}
			& [\mathbb P^1(\widehat I)]^2
			& \mathbb P^0(\widehat I)\\
			\text{quadrilateral}
			& [\mathbb Q^1(\widehat I)]^2
			& \mathbb P^1(\widehat I)\subset\mathbb Q^1(\widehat I),
		\end{array}\\[2pt]
		\Psi_e\in[\mathbb P^1(\widehat e)]^2,
		\qquad |J_e|\in\mathbb P^0(\widehat e),
		\qquad \bm n_e\in[\mathbb P^0(\widehat e)]^2.
	\end{gathered}
	\]
	Since $\bm\phi\in[\mathbb P^2]^6$,
	$\nabla\bm\phi\in[\mathbb P^1]^6$, and
	$v_h\in\mathbb P^2$, the mapped integrands satisfy
	\[
	\begin{aligned}
		\text{triangle:}\quad
		&\bigl[(\nabla\bm\phi\,\bm q\,v_h)\circ\Psi_I\bigr]|J_I|
		\in[\mathbb P^{3}(\widehat I)]^6
		\subset[\mathbb P^4(\widehat I)]^6,\\
		\text{quadrilateral:}\quad
		&\bigl[(\nabla\bm\phi\,\bm q\,v_h)\circ\Psi_I\bigr]|J_I|
		\in[\mathbb Q^{4}(\widehat I)]^6
		\subset[\mathbb Q^5(\widehat I)]^6,\\
		\text{edge:}\quad
		&\bigl[(v_h(\bm q\cdot\bm n_e)\bm\phi)\circ\Psi_e\bigr]|J_e|
		\in[\mathbb P^{4}(\widehat e)]^6
		\subset[\mathbb P^5(\widehat e)]^6.
	\end{aligned}
	\]
	The six-point triangular, the $4\times4$ tensor-product Gauss--Lobatto,
	and the four-point edge rules are exact for $\mathbb P^4(\widehat I)$, $\mathbb Q^5(\widehat I)$, and $\mathbb P^5(\widehat e)$, respectively. Hence the first two identities hold, and integration by parts gives
	\[
	\int_{I(t)}\nabla\bm\phi\bm qv_hd\bm x
	-\int_{\partial I(t)}v_h(\bm q\cdot\bm n)\bm\phi dS
	=-\int_{I(t)}\bm\phi(\bm q\cdot\nabla v_h)d\bm x.
	\]
\end{proof}
\subsubsection{Fully discrete GCL property}
We next establish the fully discrete GCL property of the CE-ALE-DG scheme.
\begin{Prop}
	The fully discrete CE-ALE-DG scheme satisfies the GCL, i.e.,
	\[
	u_h(\cdot,t_n)\equiv C
	\quad\Longrightarrow\quad
	u_h^{Cor}(\cdot,t_{n+1})\equiv C.
	\]
\end{Prop}
\begin{proof}
	The proof consists of two parts. First, we show that the SSPRK3 coefficient update preserves the constant state. Let $\bm c(t_n)=C\bm\alpha$, where $\bm\phi^T(t_n)\bm\alpha=1$. Since the basis is frozen,
	\[
	\bm\phi_t=0,
	\qquad
	\bm\phi^TC\bm\alpha=C,
	\qquad
	u_h^-=u_h^+=C.
	\]
	By consistency of the numerical ALE flux,
	\[
	\widehat g(C,C;\bm\mu^g,\bm n)
	=
	\bigl(\bm F(C)-C\bm\mu^g\bigr)\cdot\bm n.
	\]
	Substitution into \eqref{semidis} gives
	\[
	\begin{aligned}
		&\mathcal R_{I,0}(M(t),C\bm\alpha)=M^{-1}(t)\Big[
		-\mathcal Q_{\partial I(t)}
		\bigl((\bm\mu^g\cdot\bm n)\bm\phi
		\bm\phi^TC\bm\alpha\bigr)
		-\mathcal Q_{\partial I(t)}
		\bigl(\widehat g(C,C;\bm\mu^g,\bm n)\bm\phi\bigr)\\
		&+\mathcal Q_{I(t)}
		\bigl(\nabla\bm\phi\bm F(C)\bigr)
		\Big]\\
		=&M^{-1}(t)\Big[
		-C\mathcal Q_{\partial I(t)}
		\bigl((\bm\mu^g\cdot\bm n)\bm\phi\bigr)
		-\mathcal Q_{\partial I(t)}
		\bigl((\bm F(C)-C\bm\mu^g)\cdot\bm n\bm\phi\bigr)
		+\mathcal Q_{I(t)}
		\bigl(\nabla\bm\phi\,\bm F(C)\bigr)
		\Big]\\
		=&M^{-1}(t)\Big[
		-\mathcal Q_{\partial I(t)}
		\bigl((\bm F(C)\cdot\bm n)\bm\phi\bigr)
		+\mathcal Q_{I(t)}
		\bigl(\nabla\bm\phi\,\bm F(C)\bigr)
		\Big]\\
		=&-M^{-1}(t)
		\int_{I(t)}\bm\phi\bigl(\bm F(C)\cdot\nabla 1\bigr)d\bm x=0,
	\end{aligned}
	\]
	where the last equality follows from
	Lemma~\ref{lem:QuadExact} with
	$v_h=1$ and $\bm q=\bm F(C)$. Consequently, all SSPRK3 stages remain equal to \(C\bm\alpha\). Using the basis transformation, the provisional solution at \(t_{n+1}\) satisfies $	\bm\phi^T(t_{n+1})\bm c(t_{n+1})=C.$
	
	Next, we show that the conservative correction also preserves the constant state. For the constant Simpson states, flux consistency gives
	\[
	\widehat g(C,C;\bm\mu^g,\bm n)
	=
	\bigl(\bm F(C)-C\bm\mu^g\bigr)\cdot\bm n.
	\]
	Hence
	\[
	\begin{aligned}
		\mathcal R_{I,1}(C\bm\alpha,t)
		&=
		-\mathcal Q_{\partial I(t)}
		\bigl(\widehat g(C,C;\bm\mu^g,\bm n)\bigr)=
		-\mathcal Q_{\partial I(t)}
		\bigl(\bm F(C)\cdot\bm n\bigr)
		+
		C\mathcal Q_{\partial I(t)}
		\bigl(\bm\mu^g\cdot\bm n\bigr)\\
		&=
		-\int_{\partial I(t)}
		\bm F(C)\cdot\bm ndS
		+
		C\mathcal Q_{\partial I(t)}
		\bigl(\bm\mu^g\cdot\bm n\bigr)\\
		&=
		-\int_{I(t)}
		\nabla\cdot\bm F(C)d\bm x
		+
		C\mathcal Q_{\partial I(t)}
		\bigl(\bm\mu^g\cdot\bm n\bigr)=
		C\mathcal Q_{\partial I(t)}
		\bigl(\bm\mu^g\cdot\bm n\bigr).
	\end{aligned}
	\]
	Since every cell vertex moves linearly in time, $	|I(t)|\in\mathbb P^2(t)$ and $
	\frac{d}{dt}|I(t)|\in\mathbb P^1(t).$
	Moreover, the Reynolds transport theorem and exact boundary
	quadrature give
	\[
	\frac{d}{dt}|I(t)|
	=
	\int_{\partial I(t)}\bm\mu^g\cdot\bm n\,dS
	=
	\mathcal Q_{\partial I(t)}
	\bigl(\bm\mu^g\cdot\bm n\bigr).
	\]
	Since Simpson’s rule is exact for temporal polynomials of degree at most three, it follows that
	\[
	\begin{aligned}
		|I(t_{n+1})|-|I(t_n)|
		=
		\frac{\Delta t}{6}\Big[
		&\mathcal Q_{\partial I(t_n)}
		(\bm\mu^g\cdot\bm n)+4\mathcal Q_{\partial I(t_{n+\frac12})}
		(\bm\mu^g\cdot\bm n)+\mathcal Q_{\partial I(t_{n+1})}
		(\bm\mu^g\cdot\bm n)
		\Big].
	\end{aligned}
	\]
	Using $\mathcal U_I^\star(t_n)=C|I(t_n)|$ in \eqref{eq:FDConservedODE} yields
	\[
	\mathcal U_I^\star(t_{n+1})=C|I(t_{n+1})|
	=\widetilde{\mathcal U}_I(t_{n+1}).
	\]
	Thus $\lambda_I=0$, and the corrected coefficient vector $\bm{c}_{Cor}(t_{n+1})$ computed by \eqref{coefConservedUpdate} remains $C\bm\alpha$. Consequently,
	\[
	u_h^{Cor}(\cdot,t_{n+1})
	=
	\bm\phi^T(t_{n+1})\bm c_{Cor}(t_{n+1})
	\equiv C,
	\]
	which completes the proof.
\end{proof}
\subsubsection{Fully discrete TPE($2$) property}
In this part, we consider the linear advection equation \eqref{eq:linearAdvEquation}. As in the fully discrete GCL analysis, we prove that the CE-ALE-DG scheme reproduces the exact solution for $u_0\in\mathbb P^2$ by showing that the SSPRK3 coefficient update preserves polynomial transport and that the correction parameter \(\lambda_I\) is zero for such solutions.
\begin{lemma}\label{lem:TPESemiDis}
	Let $\bm m(x,y)=(1,x,y,x^2,xy,y^2)^T$, and let $T:=T(t_n)\in\mathbb R^{6\times6}$ be the matrix satisfying $\bm\phi^T(t_n)T=\bm m^T.$ The matrix $A\in\mathbb R^{6\times6}$ given in Lemma~\ref{lem:NablaMAndA} satisfies $-\bm a\cdot\nabla\bm m^T=\bm m^T A.$ If $u_h=\bm m^T\bm d$ is a single-valued polynomial in $\mathbb P^2$, with coefficient vector $\bm d\in\mathbb R^6$, then its local DG coefficient vector is $\bm c=T\bm d$ and $\mathcal R_{I,0}(M(t),T\bm d) = TA\bm d$ for $t\in[t_n,t_{n+1}].$
\end{lemma}

\begin{proof}
	First, since the basis is frozen on $[t_n,t_{n+1}]$, $	\bm\phi(t)=\bm\phi(t_n), \bm\phi_t=0.$ Using $\bm\phi^T(t_n)T=\bm m^T$ together with $u_h=\bm\phi^T\bm c=\bm m^T\bm d=\bm\phi^TT\bm d,$ and the linear independence of the basis functions, we obtain $\bm c=T\bm d$.
	
	Next, for \eqref{eq:linearAdvEquation},
	$\bm F(u_h)=\bm a u_h$. Since $u_h$ is single-valued,
	flux consistency gives
	\[
	\widehat g(u_h,u_h;\bm\mu^g,\bm n)
	=
	u_h(\bm a-\bm\mu^g)\cdot\bm n.
	\]
	Substitution into \eqref{semidis} yields
	\[
	\begin{aligned}
		\mathcal R_{I,0}(M(t),\bm c)
		=M^{-1}(t)\Big[
		&-\mathcal Q_{\partial I(t)}
		\bigl(u_h(\bm a\cdot\bm n)\bm\phi\bigr)+\mathcal Q_{I(t)}
		\bigl(\nabla\bm\phi\,\bm a u_h\bigr)
		\Big].
	\end{aligned}
	\]
	By Lemma~\ref{lem:QuadExact} with $v_h=u_h$ and
	$\bm q=\bm a$,
	\[
	\mathcal R_{I,0}(M(t),\bm c)
	=
	-M^{-1}(t)
	\int_{I(t)}
	\bm\phi(\bm a\cdot\nabla u_h)d\bm x.
	\]
	Moreover,
	\[
	\begin{aligned}
		-\bm a\cdot\nabla u_h
		&=-\bm a\cdot\nabla(\bm m^T\bm d)=\bm m^TA\bm d
		=\bm\phi^TTA\bm d.
	\end{aligned}
	\]
	Therefore,
	\[
	\begin{aligned}
		\mathcal R_{I,0}(M(t),\bm c)
		&=
		M^{-1}(t)
		\int_{I(t)}\bm\phi\bm\phi^Td\bm xTA\bm d=
		M^{-1}(t)M(t)TA\bm d
		=
		TA\bm d.
	\end{aligned}
	\]
\end{proof}
\begin{lemma}\label{lem:TPEFullDis}
	In the setting of Lemma~\ref{lem:TPESemiDis}, if $\bm c(t_n)=T\bm d(t_n),$ then \eqref{disScheme} satisfies $\widetilde{\bm c}(t_{n+1})=T\bm d(t_{n+1})$. Moreover, the Simpson states satisfy
	\[
	\begin{aligned}
		\bm c^E(t^{(0)})=T\bm d(t_n),\quad\bm c^E(t^{(2)})=T\bm d(t_{n+\frac12}),\quad\bm c^E(t^{(1)})=T\bm d(t_{n+1}).
	\end{aligned}
	\]
\end{lemma}
\begin{proof}
	Let $\mathcal L=-\bm a\cdot\nabla$. Since
	$\mathcal L\mathbb P^r\subset\mathbb P^{r-1}$, we have $	\mathcal L^3q=0$ for every $ q\in\mathbb P^2.$
	Moreover, since $\mathcal L(\bm m^T\bm d)=\bm m^TA\bm d,$ it follows that $	\bm m^TA^3\bm d=0$ for every $\bm d\in\mathbb R^6.$
	The linear independence of the components of $\bm m$ therefore gives $A^3=0.$ 
	
	By Lemma~\ref{lem:NablaMAndA}, the coefficient vector $\bm d(t)$ satisfies
	\[
	\bm d'(t)=A\bm d(t),
	\qquad
	\bm d(t_{n+1})=e^{\Delta tA}\bm d(t_n).
	\]
	
	We now prove $\widetilde{\bm c}(t_{n+1})=T\bm d(t_{n+1})$. Set $	\bm d^{(0)}=\bm d(t_n), Z=\Delta tA.$
	Since $\bm c(t^{(0)})=T\bm d^{(0)}$, Lemma~\ref{lem:TPESemiDis}
	and the first SSPRK3 stage give
	\[
	\begin{aligned}
		\bm c(t^{(1)})
		&=T\bm d^{(0)}+\Delta t\,TA\bm d^{(0)}=T(\mathbf{I}+Z)\bm d^{(0)}
		=:T\bm d^{(1)}.
	\end{aligned}
	\]
	Similarly,
	\[
	\begin{aligned}
		\bm c(t^{(2)})
		&=T\left\{
		\frac34\bm d^{(0)}
		+\frac14(\mathbf{I}+Z)\bm d^{(1)}
		\right\}=T\left(\mathbf{I}+\frac12Z+\frac14Z^2\right)\bm d^{(0)}
		=:T\bm d^{(2)},
	\end{aligned}
	\]
	and
	\[
	\begin{aligned}
		\bm c(t^{(3)})
		&=T\left\{
		\frac13\bm d^{(0)}
		+\frac23(\mathbf{I}+Z)\bm d^{(2)}
		\right\}=T\left(
		\mathbf{I}+Z+\frac12Z^2+\frac16Z^3
		\right)\bm d^{(0)}.
	\end{aligned}
	\]
	Since $Z^3=0$, $\mathbf{I}+Z+\frac12Z^2+\frac16Z^3=e^Z.$ Here, $e^Z=e^{\Delta tA}$ is the exact evolution operator of
	$\bm d'(t)=A\bm d(t)$ over one time step. Hence, $e^Z\bm d(t_n)=\bm d(t_{n+1}).$
	Therefore,
	\[
	\widetilde{\bm c}(t_{n+1})
	=\bm c(t^{(3)})
	=Te^{\Delta tA}\bm d(t_n)
	=T\bm d(t_{n+1}).
	\]
	
	Finally, the endpoint Simpson states satisfy
	\[
	\bm c^E(t^{(0)})=T\bm d(t_n),
	\qquad
	\bm c^E(t^{(1)})=T\bm d(t_{n+1}).
	\]
	For the midpoint state,
	\[
	\begin{aligned}
		\bm c^E(t^{(2)})
		&=T\left(
		\frac14\bm d^{(0)}
		+\frac14\bm d^{(1)}
		+\frac12\bm d^{(2)}
		\right)\\
		&=T\left(\mathbf{I}+\frac12Z+\frac18Z^2\right)\bm d(t_n)\\
		&=Te^{Z/2}\bm d(t_n)
		=T\bm d(t_{n+\frac12}).
	\end{aligned}
	\]
\end{proof}
\begin{lemma}\label{lem:TPEAuxCorr}
	Under the hypotheses of Lemma~\ref{lem:TPEFullDis}, let $u(x,y,t)=u_0(x-a_1t,y-a_2t),$ where $ u_0\in\mathbb P^2.$ If $	\mathcal U_I^\star(t_n) = \int_{I(t_n)}u(x,y,t_n)d\bm x,$ then the update \eqref{eq:FDConservedODE} satisfies $	\mathcal U_I^\star(t_{n+1}) = \int_{I(t_{n+1})}u(x,y,t_{n+1})d\bm x.$
\end{lemma}
\begin{proof}
	Define $\mathcal U_I^{ex}(t):=\int_{I(t)}u(x,y,t)d\bm x, \mathcal S_I^{ex}(t):=\frac{d}{dt}\mathcal U_I^{ex}(t).$
	Applying the Reynolds transport theorem and using
	\eqref{eq:linearAdvEquation}, we obtain
	\[
	\begin{aligned}
		\mathcal S_I^{ex}(t)
		&=\int_{I(t)}u_td\bm x
		+\int_{\partial I(t)}u\bm\mu^g\cdot\bm ndS\\
		&=-\int_{I(t)}\nabla\cdot(\bm a u)d\bm x
		+\int_{\partial I(t)}u\bm\mu^g\cdot\bm ndS\\
		&=-\int_{\partial I(t)}
		u(\bm a-\bm\mu^g)\cdot\bm ndS.
	\end{aligned}
	\]
	To determine the temporal degree of $\mathcal S_I^{ex}(t)$, let $\widehat I$ be the reference element and let $\widehat{\bm x}=(\xi,\eta)^T\in\widehat I$ denote its reference coordinate. Define the element mapping and its Jacobian factor by
	\[
	\Psi_I(\widehat{\bm x},t):\widehat I\to I(t),
	\qquad
	\bm x=\Psi_I(\widehat{\bm x},t),
	\qquad
	|J_I(\widehat{\bm x},t)|
	:=\left|\det D_{\widehat{\bm x}}\Psi_I(\widehat{\bm x},t)\right|.
	\]
	The mapping is affine for triangular cells and bilinear for quadrilateral cells. Since every vertex trajectory is linear on
	$[t_n,t_{n+1}]$,
	\[
	\Psi_I(\widehat{\bm x},t)
	=\Psi_I(\widehat{\bm x},t_n)
	+(t-t_n)\bm\mu_I^g(\widehat{\bm x}).
	\]
	Hence, for each fixed $\widehat{\bm x}$,
	\[
	\begin{aligned}
		&\Psi_I(\widehat{\bm x},t)\in[\mathbb P^1(t)]^2,\quad
		D_{\widehat{\bm x}}\Psi_I(\widehat{\bm x},t)
		\in[\mathbb P^1(t)]^{2\times2},\\
		&u(\Psi_I(\widehat{\bm x},t),t)\in\mathbb P^2(t),\quad
		|J_I(\widehat{\bm x},t)|\in\mathbb P^2(t).
	\end{aligned}
	\]
	Consequently,
	\[
	\mathcal U_I^{ex}(t)
	=\int_{\widehat I}
	u(\Psi_I(\widehat{\bm x},t),t)
	|J_I(\widehat{\bm x},t)|d\widehat{\bm x}
	\in\mathbb P^4(t),\quad
	\mathcal S_I^{ex}\in\mathbb P^3(t).
	\]
	
	Next, we show that $\mathcal R_{I,1}$ evaluated with the Simpson states agrees with $\mathcal S_I^{ex}$ at the three corresponding times. For $t\in\{t_n,t_{n+\frac12},t_{n+1}\}$ and every straight edge $e(t)$, linear interpolation of the vertex velocities gives
	\[
	\begin{aligned}
		u(\cdot,t)|_{e(t)}&\in\mathbb P^2(e(t)),\quad \bm\mu^g(\cdot,t)|_{e(t)}\in[\mathbb P^1(e(t))]^2, \quad\bm n_e\in[\mathbb P^0(e(t))]^2.
	\end{aligned}
	\]
	Consequently,
	\[
	u(\cdot,t)(\bm a-\bm\mu^g(\cdot,t))\cdot\bm n_e
	\in\mathbb P^3(e(t)).
	\]
	The edge mapping is affine, so the four-point Gauss--Lobatto rule, which
	is exact through degree five, gives
	\[
	\mathcal Q_{\partial I(t)}
	\left(u(\cdot,t)(\bm a-\bm\mu^g)\cdot\bm n\right)
	=
	\int_{\partial I(t)}
	u(\cdot,t)(\bm a-\bm\mu^g)\cdot\bm n\,dS.
	\]
	Flux consistency and Lemma~\ref{lem:TPEFullDis} therefore yield
	\[
	\begin{aligned}
		\mathcal R_{I,1}(\bm c^E(t^{(0)}),t_n)
		&=\mathcal S_I^{ex}(t_n),\\
		\mathcal R_{I,1}(\bm c^E(t^{(2)}),t_{n+\frac12})
		&=\mathcal S_I^{ex}(t_{n+\frac12}),\\
		\mathcal R_{I,1}(\bm c^E(t^{(1)}),t_{n+1})
		&=\mathcal S_I^{ex}(t_{n+1}).
	\end{aligned}
	\]
	Finally, since $\mathcal S_I^{ex}\in\mathbb P^3(t)$, Simpson's rule is exact and gives
	\[
	\int_{t_n}^{t_{n+1}}\mathcal S_I^{ex}(t)\,dt
	=\frac{\Delta t}{6}\left[
	\mathcal S_I^{ex}(t_n)
	+4\mathcal S_I^{ex}(t_{n+\frac12})
	+\mathcal S_I^{ex}(t_{n+1})
	\right].
	\]
	Substitution into \eqref{eq:FDConservedODE} yields
	\[
	\begin{aligned}
		\mathcal U_I^\star(t_{n+1})
		&=\mathcal U_I^{ex}(t_n)
		+\int_{t_n}^{t_{n+1}}\mathcal S_I^{ex}(t)\,dt=\mathcal U_I^{ex}(t_n)
		+\int_{t_n}^{t_{n+1}}\frac{d}{dt}\mathcal U_I^{ex}(t)\,dt
		=\mathcal U_I^{ex}(t_{n+1}).
	\end{aligned}
	\]
\end{proof}
\begin{Prop}
	The fully discrete CE-ALE-DG scheme satisfies TPE(2). If $	u_h(\bm x,t_n)=u(\bm x,t_n)\in\mathbb P^2$ and $ \mathcal U_I^\star(t_n)=\int_{I(t_n)}u(\bm x,t_n)d\bm x,$
	then $u_h^{Cor}(\bm x,t_{n+1})=u(\bm x,t_{n+1}).$
\end{Prop}

\begin{proof}
	By Lemmas~\ref{lem:TPESemiDis} and \ref{lem:TPEFullDis} and the basis transformation, $\bm\phi^T(t_{n+1})\bm c(t_{n+1})=u(\bm x,t_{n+1})$. Hence, by Lemma~\ref{lem:TPEAuxCorr}, $	\widetilde{\mathcal U}_I(t_{n+1}) =\int_{I(t_{n+1})}ud\bm x =\mathcal U_I^\star(t_{n+1}).$ Therefore, $\lambda_I=0$, and \eqref{coefConservedUpdate} gives $	\bm c_{Cor}(t_{n+1})=\bm c(t_{n+1}).$ Consequently,
	\[
	\begin{aligned}
		u_h^{Cor}(\cdot,t_{n+1}) =\bm\phi^T(t_{n+1})\bm c_{Cor}(t_{n+1})=\bm\phi^T(t_{n+1})\bm c(t_{n+1})=u(\cdot,t_{n+1}).
	\end{aligned}
	\]
	Thus, the fully discrete CE-ALE-DG scheme exactly transports every
	polynomial in $\mathbb P^2$, which proves the TPE($2$) property.
\end{proof}
\begin{remark}
	It is worth noting that the CE-ALE-DG scheme is not restricted to the third order. Such an extension requires the following modifications. First, the DG approximation space is extended from $\mathbb P^2$ to $\mathbb P^k$ $(k\geq 3)$. Then, the SSPRK3 method is replaced by a higher-order Runge--Kutta method. Finally, Simpson's rule and the spatial quadrature are replaced by rules of sufficient accuracy. The form of the conservative correction remains unchanged; with compatible higher-order temporal and spatial quadrature, it also vanishes for exactly transported polynomials.
\end{remark}

Taken together, these results resolve the main compatibility challenge between strict conservation and exact polynomial transport. Under the stated mesh-motion and quadrature conditions, the fully discrete CE-ALE-DG scheme is locally and globally conservative while preserving the GCL and TPE($2$) on both triangular and quadrilateral moving meshes.
\subsection{Discontinuity treatment}
To suppress spurious oscillations near discontinuities, we combine the KXRCF troubled-cell indicator~\cite{krivodonova2004shock} with the oscillation-eliminating (OE) procedure~\cite{peng2025oedg}. At each SSPRK3 stage, troubled cells are detected by the KXRCF indicator and treated with the OE procedure. To verify that this stagewise treatment does not destroy the TPE($k$) property established above, we first introduce the notion of a $k$-exact limiter.
\begin{Def}[$k$-exact limiter]
	Let $\Lambda_h$ be a limiter acting on the DG space. For a global polynomial $p\in\mathbb P^k$, let $u_h^p$ denote its DG representation determined by
	\[
	\int_{I(t)} u_h^p q\,d\bm x
	=
	\int_{I(t)} p q\,d\bm x,
	\quad
	\forall q\in\mathbb P^k(I(t)),\quad
	\forall I(t)\in\mathcal T_h(t).
	\]
	The limiter $\Lambda_h$ is called $k$-exact if $	\Lambda_h(u_h^p)=u_h^p, \forall p\in\mathbb P^k.$
	\end{Def}
\begin{proposition} 
	Assume that the unlimited Runge--Kutta DG scheme satisfies TPE($k$) and that all stage solutions remain  in $\mathbb P^k$. If $\Lambda_h$ is a $k$-exact limiter, then applying $\Lambda_h$ after each stage preserves TPE($k$).
\end{proposition}
\begin{proof} 
	Since all stage solutions lie in $\mathbb P^k$, they are unchanged by the $k$-exact limiter. Hence, the limited and unlimited updates coincide, and the limited scheme inherits TPE($k$).
\end{proof}
\begin{remark}
	Total variation bounded (TVB) \cite{cockburn1989tvb}, and weighted essentially non-oscillatory (WENO) limiters \cite{qiu2005weno} do not generally guarantee $k$-exactness, whereas the oscillation-free (OF) \cite{lu2021ofdg} and OE procedures \cite{peng2025oedg} are $k$-exact. Since KXRCF is inactive and OE acts as the identity on $\mathbb P^2$, the KXRCF--OE limiter is $2$-exact and therefore preserves TPE($2$).
\end{remark}

For clarity, Figures~\ref{fig:KXRCF-OE}
and~\ref{fig:scheme-summary} summarize the KXRCF--OE procedure and the complete CE-ALE-DG scheme, respectively.
\begin{figure}[htbp]
	\centering
	\begin{minipage}[b]{0.39\textwidth}
		\centering
		\begin{tikzpicture}[
			yscale=1.2,
			every node/.style={font=\normalsize},
			box/.style={
				rectangle,
				draw,
				rounded corners,
				align=center,
				minimum width=2.45cm,
				minimum height=0.60cm,
				font=\normalsize
			},
			decision/.style={
				diamond,
				draw,
				aspect=1.9,
				align=center,
				inner sep=1pt,
				font=\normalsize
			},
			arrow/.style={->, thick}
			]
			\node[box] (stage) at (0,0)
			{SSPRK3-stage solution};
			\node[box] (indicator) at (0,-0.80)
			{KXRCF indicator};
			\node[decision] (troubled) at (0,-1.95)
			{Troubled cell?};
			\node[box] (oe) at (0,-3.10)
			{OE procedure};
			\node[box] (output) at (0,-3.90)
			{Stage coefficients};
			
			\draw[arrow] (stage) -- (indicator);
			\draw[arrow] (indicator) -- (troubled);
			\draw[arrow] (troubled.south)
			-- node[right] {yes} (oe.north);
			\draw[arrow] (oe) -- (output);
			\draw[arrow] (troubled.west)
			-- node[above,inner sep=1pt] {no}
			++(-1.2,0)
			|- (output.west);
		\end{tikzpicture}
		\caption{The KXRCF-OE procedure at one SSPRK3 stage.}
		\label{fig:KXRCF-OE}
	\end{minipage}
	\hfill
	\begin{minipage}[b]{0.58\textwidth}
		\centering
		\begin{tikzpicture}[
			xscale=1.1,
			yscale=1.1,
			every node/.style={font=\normalsize},
			box/.style={
				rectangle,
				draw,
				rounded corners,
				align=center,
				minimum width=2.1cm,
				minimum height=0.65cm,
				font=\normalsize
			},
			arrow/.style={->, thick}
			]
			\def\radius{2.35cm}
			
			\node[box] (start) at (90:\radius)
			{Solution at $t_n$};
			
			\node[box] (basis) at (18:\radius)
			{Freeze basis};
			
			\node[box] (rk) at (-52:\radius)
			{SSPRK3 with\\KXRCF-OE \\procedure};
			
			\node[box] (correction) at (-126:\radius)
			{Basis update};
			
			\node[box] (update) at (162:\radius)
			{Conservative\\correction};
			
			\node[align=center,font=\normalsize\bfseries] at (0,0)
			{CE-ALE-DG\\scheme};
			
			\draw[arrow] (start) -- (basis);
			\draw[arrow] (basis) -- (rk);
			\draw[arrow] (rk) -- (correction);
			\draw[arrow] (correction) -- (update);
			\draw[arrow,dashed] (update) --
			node[
			midway,
			above,
			sloped,
			font=\normalsize,
			yshift=-2pt
			]
			{$t_{n+1}$}
			(start);
		\end{tikzpicture}
		\caption{The third-order CE-ALE-DG scheme from
			$t_n$ to $t_{n+1}$.}
		\label{fig:scheme-summary}
	\end{minipage}
	\vspace{-20pt}
\end{figure}
\section{Numerical examples}\label{sec:Num_examples} In this section, we present numerical examples to assess the accuracy, conservation, GCL, and TPE properties of the proposed third-order CE-ALE-DG scheme, together with its robustness for discontinuous problems. Both structured triangular and quadrilateral meshes are considered.

On triangular meshes, the proposed scheme is compared with the ALE-DG scheme developed by Fu et~al.~\cite{fu2019simplex}, hereafter referred to as the ALE-DG-FSX scheme. This scheme is selected as the reference method because it is representative of existing direct ALE-DG formulations and because its error analysis relies on a mesh-independent Lipschitz bound on the mesh velocity. Comparing the two schemes under random mesh motion without imposing this bound provides a direct numerical test of whether TPE($k$) can prevent mesh-motion-induced loss of convergence order. Since the ALE-DG-FSX scheme was originally designed for moving triangular meshes, the comparisons are restricted to triangular meshes. On quadrilateral meshes, only results obtained with the CE-ALE-DG scheme are reported to demonstrate the applicability of the proposed framework to quadrilateral elements.

Following~\cite{li2019high}, we use $\mathrm{Lip}(\bm\mu^g)$ to quantify the spatial variation of the mesh velocity. The large values observed in our experiments, together with their growth under mesh refinement, indicate strong spatial variation and the absence of a mesh-independent derivative bound. The time-step size is determined by
\[
\Delta t_n
=
\frac{C_{\mathrm{CFL}}}{2\alpha_{\max}^n}
\sqrt{\min_{I\in\mathcal T_h(t_n)}|I|},
\]
where $\alpha_{\max}^n$ denotes the maximum ALE characteristic speed over all cell edges, and the CFL number $C_{\mathrm{CFL}}$ is chosen as $\frac{1}{2k+1}$, yielding $C_{\mathrm{CFL}}=0.2$ for the third-order CE-ALE-DG scheme with $k=2$.
\subsection{Mesh motion}
To examine the scheme under mesh motion without a uniform Lipschitz bound on the mesh velocity, we employ the randomly moving mesh introduced in~\cite{cai2026beyond}. Define the characteristic initial mesh size by $h_0:= \sqrt{\min_{I\in\mathcal T_h(0)}|I|}.$ The position of the $i$-th mesh vertex at time $t_n$ is prescribed by
\[
\bm x_i(t_n)
=
\bm x_{i,0}
+t_n\bm T_r
+
\begin{cases}
	h_0C_r\bm N_i^n, & \bm x_i\ \text{is an interior vertex},\\
	\bm 0,           & \bm x_i\ \text{is a boundary vertex},
\end{cases}
\]
where $\bm x_{i,0}$ is the initial vertex position, and $C_r$ controls the magnitude of the random perturbation. The components of $\bm N_i^n\in\mathbb R^2$ are independently and uniformly distributed in $[-0.5,0.5]$ for different interior vertices and time levels. We set $\bm N_i^0=\bm 0$ and use the fixed random seed \texttt{std::mt19937 gen(123456789u)} for reproducibility. In the TPE and accuracy tests, $C_r=0.5$ and $\bm T_r=(-0.48249510,-0.87589867)^T$ are used.

To illustrate how the mesh-velocity Lipschitz constraint suppresses the magnitude of mesh motion, Figure~\ref{fig:mesh_overlay} compares the local meshes obtained with and without this constraint under the same underlying random mesh-velocity field. The constrained mesh remains close to the initial mesh, whereas the unconstrained mesh undergoes visibly larger deformation. The unconstrained case therefore reflects the original random mesh motion without modification by the Lipschitz constraint.
\begin{figure}[htbp]
	\vspace{-10pt}
	\centering
	\captionsetup{skip=2pt}
	\includegraphics[width=1.0\textwidth]
	{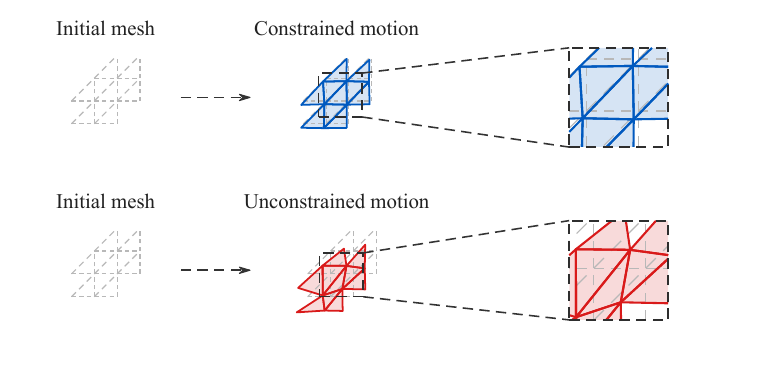}
	\caption{Comparison of local mesh configurations at $t=0.04$ under the same underlying random mesh motion, with and without the mesh-velocity Lipschitz constraint. The gray dashed lines denote the initial mesh, while the blue and red solid lines show the meshes obtained with and without the constraint, respectively. The black dashed boxes and connecting lines indicate the regions shown in the enlarged views.}
	\label{fig:mesh_overlay}
	\vspace{-15pt}
\end{figure}

We also consider a Lagrangian-type mesh velocity to demonstrate the practical applicability of the CE-ALE-DG scheme. The mesh velocity is chosen according to the local fluid velocity~\cite{badwaik2020single,CHENG20071567} and is used in the nonsmooth tests of Section~\ref{sec:nonsmooth_tests}.
\subsection{TPE tests}
We first verify that the CE-ALE-DG scheme satisfies TPE($k$) for $k=0,1,2$ using the linear advection equation with the quadratic DG space $\mathbb V_I^2(t)$. Here, TPE(0) coincides with the GCL. In this test, the flux function is given by $\bm F(u)=\bm a u,$ where $\bm a=(0.43,-0.59)^T$ is a constant advection velocity. The exact solution is $u(x,y,t)=u_0(x-0.43t,y+0.59t).$ The computation is performed on the domain $\Omega(0)=[0,1]\times[0,1]$ up to $t=1$, using $40\times40$ structured triangular and quadrilateral meshes, with exact boundary conditions imposed. The initial condition is chosen as
\[
u_0(x,y)
=
\gamma_0+\gamma_1x+\gamma_2y+\gamma_3x^2+\gamma_4xy+\gamma_5y^2.
\]
For each TPE($k$) test, the coefficients of terms of degree greater
than $k$ are set to zero, while the remaining coefficients are sampled
uniformly from $[-5,5]$. The fixed random seed \texttt{std::mt19937 gen(2026999999u)} is used for reproducibility.
\begin{table}[htbp]
	\centering
	\scriptsize
	\caption{Errors in the TPE($k$) tests for the CE-ALE-DG and ALE-DG-FSX schemes on randomly moving meshes.}
	\label{tab:tpe-test}
	\begin{tabular}{c c c c c}
		\hline
		& Tests & TPE(0) & TPE(1) & TPE(2) \\
		\hline
		CE-ALE-DG & $L^1$ error
		& 1.99e-14 & 1.07e-13 & 2.51e-13 \\
		Tri mesh& $L^\infty$ error
		& 7.06e-14 & 7.57e-13 & 1.71e-12 \\
		\hline
		CE-ALE-DG & $L^1$ error
		& 1.26e-14 & 6.79e-14 & 1.57e-13 \\
		Quad mesh& $L^\infty$ error
		& 4.11e-14 & 4.29e-13 & 1.02e-12 \\
		\hline
		ALE-DG-FSX & $L^1$ error
		& 5.30e-16 & 2.29e-06 & 8.34e-06 \\
		Tri mesh& $L^\infty$ error
		& 3.11e-15 & 1.33e-05 & 8.37e-05 \\
		\hline
	\end{tabular}
\end{table}

Table~\ref{tab:tpe-test} shows that the errors of the CE-ALE-DG scheme remain near the accumulated roundoff level for $k=0,1,2$ on both triangular and quadrilateral meshes, confirming its discrete GCL, TPE(1), and TPE(2) properties. In contrast, on the triangular mesh, the ALE-DG-FSX scheme preserves only TPE(0) but does not preserve TPE(1) or TPE(2). The CE-ALE-DG results are consistent with the analysis in Section~\ref{sec:CE-ALE-DG}.
\subsection{Accuracy tests}
In this subsection, we examine the accuracy and conservation properties of the CE-ALE-DG scheme under random mesh motion without an imposed mesh-velocity Lipschitz constraint. Two problems are considered: Burgers' equation and the compressible Euler equations. For both problems, CE-ALE-DG results are reported on structured triangular and quadrilateral meshes, while comparisons with the ALE-DG-FSX scheme are performed on triangular meshes. Periodic boundary conditions are imposed in all tests.

The CE-ALE-DG and ALE-DG-FSX schemes without the imposed Lipschitz constraint are denoted by CE-NoLip and FSX-NoLip, respectively. Following~\cite{li2019high}, FSX-Lip denotes the ALE-DG-FSX scheme with both the mesh-velocity magnitude and its discrete Lipschitz estimate bounded by $10$.
\subsubsection{Burgers' equation}
For Burgers' equation, the flux function is given by $\bm F(u)=\left(0.5u^2,0.5u^2\right)^T.$ The initial condition is $u_0(x,y)=1+0.5\sin(2\pi(x+y)).$ The computation is performed on the domain
$\Omega(0)=[0,1]\times[0,1]$ up to $t=0.1$.
\begin{table}[htbp]
	\centering
	\caption{Errors and convergence orders for Burgers’ equation.}
	\label{tab:burgersSm-test}
	\scriptsize
	\begin{tabular}{c c c c c c c}
		\hline
		& $N_x \times N_y$ & $20\times20$ & $40\times40$ & $80\times80$ & $160\times160$ & $320\times320$  \\
		\hline
		& $L^1$ error
		& 1.7183e-03 & 2.0574e-04 & 2.0974e-05
		& 2.3264e-06 & 2.7190e-07 \\
		& Order
		& -- & 3.06 & 3.29 & 3.17 & 3.10 \\
		CE-NoLip & $L^\infty$ error
		& 2.5905e-02 & 4.8934e-03 & 5.7161e-04
		& 7.6291e-05 & 1.0175e-05 \\
		Tri mesh & Order
		& -- & 2.40 & 3.10 & 2.91 & 2.91 \\
		& $\mathrm{Lip}(\bm\mu^g)$
		& 325.97 & 714.76 & 1508.59 & 3233.28 & 6853.40 \\
		\hline
		& $L^1$ error
		& 2.8912e-03 & 3.3381e-04 & 2.8187e-05
		& 2.3747e-06 & 2.3174e-07 \\
		& Order
		& -- & 3.11 & 3.57 & 3.57 & 3.36 \\
		CE-NoLip & $L^\infty$ error
		& 4.1197e-02 & 7.5122e-03 & 1.2272e-03
		& 1.7563e-04 & 2.3022e-05 \\
		Quad mesh & Order
		& -- & 2.46 & 2.61 & 2.80 & 2.93 \\
		& $\mathrm{Lip}(\bm\mu^g)$
		& 386.47 & 796.02 & 1659.58 & 3488.30 & 7954.85 \\
		\hline
		& $L^1$ error
		& 5.9086e-04 & 9.0114e-05 & 1.3066e-05
		& 2.3988e-06 & 7.6287e-07 \\
		& Order
		& -- & 2.71 & 2.79 & 2.45 & 1.65 \\
		FSX-NoLip & $L^\infty$ error
		& 1.1324e-02 & 2.5250e-03 & 4.0524e-04
		& 6.6256e-05 & 1.3287e-05 \\
		Tri mesh & Order
		& -- & 2.16 & 2.64 & 2.61 & 2.32 \\
		& $\mathrm{Lip}(\bm\mu^g)$
		& 325.89 & 714.76 & 1508.58 & 3233.28 & 6853.40 \\
		\hline
		& $L^1$ error
		& 5.8517e-04 & 7.6870e-05 & 9.1679e-06
		& 1.1141e-06 & 1.3827e-07 \\
		& Order
		& -- & 2.93 & 3.07 & 3.04 & 3.01 \\
		FSX-Lip & $L^\infty$ error
		& 9.5871e-03 & 2.2737e-03 & 3.0827e-04
		& 4.1836e-05 & 5.3469e-06 \\
		Tri mesh & Order
		& -- & 2.08 & 2.88 & 2.88 & 2.97 \\
		& $\mathrm{Lip}(\bm\mu^g)$
		& 10.00 & 10.00 & 10.00 & 10.00 & 10.00 \\
		\hline
	\end{tabular}
\end{table}

Table~\ref{tab:burgersSm-test} shows that CE-ALE-DG retains third-order accuracy on both triangular and quadrilateral meshes without a uniform mesh-velocity Lipschitz bound. The ALE-DG-FSX scheme loses this accuracy on triangular meshes but recovers it when the Lipschitz constraint is imposed, demonstrating the advantage of TPE-based coefficient evolution.
\subsubsection{Compressible Euler equations}
We next consider the two-dimensional compressible Euler equations $\bm u_t+\nabla\cdot\bm F(\bm u)=0,$
where $\bm u=(\rho,\rho \mu,\rho \nu,E)^T, \bm F(\bm u) =
\left[
\rho\bm \mu,\,
\rho\bm \mu\otimes\bm \mu+p\mathbf{I},\,
(E+p)\bm \mu
\right]^T.$
Here $\rho$ is the density, $\bm \mu=(\mu,\nu)^T$ is the fluid velocity, $E$ is the total energy, and $\mathbf{I}$ is the identity matrix. The pressure is given by $p=(\gamma-1)\left(E-\frac12\rho|\bm \mu|^2\right),$ with $\gamma=1.4$. The smooth initial condition is chosen as
{
	\setlength{\abovedisplayskip}{3pt}
	\setlength{\belowdisplayskip}{3pt}
	\[
	\rho(x,y,0)=1+0.2\sin(2\pi(x+y)),\quad
	\mu(x,y,0)=0.7,\quad
	\nu(x,y,0)=0.3,\quad
	p(x,y,0)=1.
	\]
}
The computation is performed on $\Omega(0)=[0,1]\times[0,1]$ up to $t=0.1$.
\begin{table}[htbp]
	\centering
	\caption{Density errors and convergence orders for the compressible Euler equations.}
	\label{tab:EulerSm-test}
	\scriptsize
	\begin{tabular}{c c c c c c c}
		\hline
		& $N_x \times N_y$ & $20\times20$ & $40\times40$ & $80\times80$ & $160\times160$ & $320\times320$  \\
		\hline
		& $L^1$ error
		& 1.1610e-04 & 1.5537e-05 & 1.9570e-06
		& 2.4530e-07 & 3.0866e-08 \\
		& Order
		& -- & 2.90 & 2.99 & 3.00 & 2.99 \\
		CE-NoLip
		& $L^\infty$ error
		& 3.7034e-04 & 5.9018e-05 & 7.5821e-06
		& 1.0270e-06 & 1.3139e-07 \\
		Tri mesh
		& Order
		& -- & 2.65 & 2.96 & 2.88 & 2.97 \\
		& $\mathrm{Lip}(\bm{\mu}^g)$
		& 393.64 & 865.21 & 1840.90 & 3924.44 & 8234.84 \\
		\hline
		& $L^1$ error
		& 8.4975e-05 & 1.0316e-05 & 1.2500e-06
		& 1.5383e-07 & 1.9162e-08 \\
		& Order
		& -- & 3.04 & 3.04 & 3.02 & 3.01 \\
		CE-NoLip & $L^\infty$ error
		& 6.1886e-04 & 9.4864e-05 & 1.3222e-05
		& 1.6321e-06 & 2.1493e-07 \\
		Quad mesh & Order
		& -- & 2.71 & 2.84 & 3.02 & 2.92 \\
		& $\mathrm{Lip}(\bm\mu^g)$
		& 446.62 & 917.24 & 1933.18 & 4076.47 & 8941.29 \\
		\hline
		& $L^1$ error
		& 1.0010e-04 & 1.3405e-05 & 2.4963e-06
		& 9.6498e-07 & 4.3683e-07 \\
		& Order
		& -- & 2.90 & 2.42 & 1.37 & 1.14 \\
		FSX-NoLip & $L^\infty$ error
		& 3.4724e-04 & 5.5261e-05 & 1.1044e-05
		& 3.7097e-06 & 1.7857e-06 \\
		Tri mesh & Order
		& -- & 2.65 & 2.32 & 1.57 & 1.05 \\
		& $\mathrm{Lip}(\bm\mu^g)$
		& 393.63 & 865.21 & 1840.90 & 3924.44 & 8234.84 \\
		\hline
		& $L^1$ error
		& 8.3341e-05 & 1.0702e-05 & 1.3046e-06
		& 1.5629e-07 & 1.8834e-08 \\
		& Order
		& -- & 2.96 & 3.04 & 3.06 & 3.05 \\
		FSX-Lip & $L^\infty$ error
		& 2.3047e-04 & 2.7893e-05 & 3.4212e-06
		& 4.1640e-07 & 5.1193e-08 \\
		Tri mesh & Order
		& -- & 3.05 & 3.03 & 3.04 & 3.02 \\
		& $\mathrm{Lip}(\bm\mu^g)$
		& 10.00 & 10.00 & 10.00 & 10.00 & 10.00 \\
		\hline
	\end{tabular}
\end{table}

Table~\ref{tab:EulerSm-test} shows the same trend: CE-ALE-DG retains third-order accuracy on both meshes without a uniform Lipschitz bound, whereas ALE-DG-FSX on triangular meshes requires the Lipschitz constraint. Together with Table~\ref{tab:burgersSm-test}, these results provide numerical evidence that TPE(2) helps prevent mesh-motion-induced order reduction in nonlinear problems.
\begin{table}[htbp]
	\centering
	\caption{Conservation errors of the CE-ALE-DG scheme
		on triangular and quadrilateral meshes.}
	\label{tab:conservation-error}
	\scriptsize
	\setlength{\tabcolsep}{3pt}
	\begin{tabular}{c c c c c c c}
		\hline
		Equation & Mesh
		& $20\times20$ & $40\times40$ & $80\times80$
		& $160\times160$ & $320\times320$ \\
		\hline
		\multirow{2}{*}{Burgers' equation}
		& Tri
		& 2.4758e-14 & 7.5606e-14 & 6.2172e-14
		& 1.2523e-13 & 1.4404e-12 \\
		\cline{2-7}
		& Quad
		& 6.6613e-15 & 5.2847e-14 & 3.5305e-14
		& 2.0095e-13 & 1.6784e-12 \\
		\hline
		\multirow{2}{*}{Euler equations}
		& Tri
		& 3.5051e-14 & 1.3067e-13 & 1.8818e-13
		& 4.9614e-13 & 1.4895e-12 \\
		\cline{2-7}
		& Quad
		& 3.5018e-14 & 7.1468e-14 & 3.2008e-13
		& 8.6717e-13 & 1.5463e-12 \\
		\hline
	\end{tabular}
\end{table}

Table~\ref{tab:conservation-error} also shows machine-precision conservation errors for both problems on triangular and quadrilateral meshes. For the Euler equations, the reported error is the maximum over the four conserved variables.
\subsection{Nonsmooth tests}\label{sec:nonsmooth_tests} In this subsection, we consider several compressible Euler problems with nonsmooth solutions to assess the numerical robustness of the CE-ALE-DG scheme. The KXRCF-OE procedure is used to suppress spurious oscillations, while a positivity-preserving limiter~\cite{FU2022111600,zhang2010positivity} is applied to maintain positive density and pressure. Throughout this subsection, Ref denotes the reference solution; Rand and Lag correspond to results obtained with random and Lagrangian mesh velocities, and Con and Non-con indicate whether the conservative correction is applied or not, respectively.
\subsubsection{Shu--Osher problem}
We first consider the Shu--Osher problem~\cite{shu1989efficient}, which is posed here in a quasi-one-dimensional setting and solved on two-dimensional moving meshes. For the random mesh velocity, the computational domain is $\Omega(0)=[-5,5]\times[-0.125,0.125]$, using a $400\times10$ structured quadrilateral mesh. For the Lagrangian mesh velocity, we use the enlarged domain $\Omega(0)=[-9.75,5]\times[-0.125,0.125]$ with a $590\times10$ structured quadrilateral mesh, so that the mesh size in the $x$-direction remains the same in both cases. The initial condition is {\setlength{\abovedisplayskip}{3pt}
	\setlength{\belowdisplayskip}{3pt}
	\[
	(\rho,\mu,\nu,p,\gamma)(x,y)=
	\begin{cases}
		(3.857143,2.629369, 0, 10.33333,1.4), & x\leq-4,\\
		(1+0.2\sin(5x), 0, 0, 1,1.4), & x>-4.
	\end{cases}
	\]
}The final time is $t=1.8$. An inflow condition is imposed at the left boundary and an outflow condition at the right boundary, while periodic boundary conditions are used on the top and bottom boundaries.
\begin{figure}[htbp]
	\centering
	\captionsetup{skip=2pt}
	
	\includegraphics[width=0.35\textwidth]
	{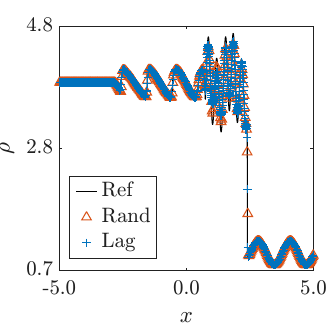}
	\includegraphics[width=0.35\textwidth]
	{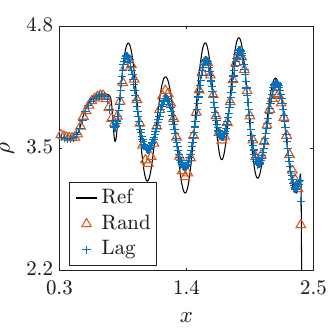}
	\caption{Density figures for the Shu--Osher problem under random (Rand) and Lagrangian (Lag) mesh motions: full domain (left) and zoomed view (right).}
	\label{fig:Shu}
	
	\vspace{-8pt}
	
	\includegraphics[width=0.35\textwidth]
	{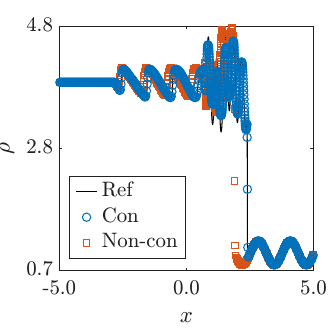}
	\includegraphics[width=0.35\textwidth]
	{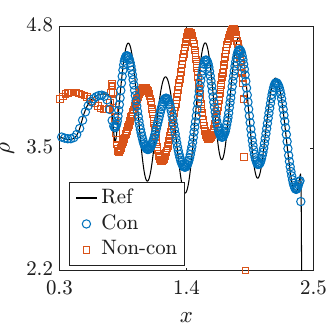}
	\caption{Density figures for the Shu--Osher problem under Lagrangian mesh motion with (Con) and without (Non-con) the conservative correction: full domain (left) and zoomed view (right).}
	\label{fig:ShuCon}
	\vspace{-28pt}
\end{figure}

Figure~\ref{fig:Shu} compares the reference solution with the CE-ALE-DG results obtained using random ($C_r=0.5$ and $\bm T_r=(0,0)^T$) and Lagrangian mesh velocities. Both results capture the
shock location and the oscillatory density waves behind the shock. The zoomed view shows no substantial spurious oscillations in the limited solutions, while the main wave structures remain well preserved. Figure~\ref{fig:ShuCon} illustrates the role of the conservative correction by comparing the CE-ALE-DG result with that obtained when the correction is omitted under the same Lagrangian mesh-motion construction. Without the correction, substantial phase and amplitude errors appear in the post-shock density waves, whereas the CE-ALE-DG result remains closer to the reference solution.
\subsubsection{Woodward--Colella blast wave problem} We next consider the Woodward--Colella blast wave problem~\cite{woodward1984numerical} in a quasi-one-dimensional setting. The computational domain is $\Omega(0)=[0,1]\times[-0.0125,0.0125],$ using a $400\times10$ structured quadrilateral mesh. The initial condition is given by
{\setlength{\abovedisplayskip}{3pt}
	\setlength{\belowdisplayskip}{3pt}
	\[
	(\rho,\mu,\nu,p,\gamma)(x,y)=
	\begin{cases}
		(1, 0, 0, 1000, 1.4), & x<0.1,\\
		(1, 0, 0, 0.01, 1.4), & 0.1\le x<0.9,\\
		(1, 0, 0, 100, 1.4), & x\ge 0.9.
	\end{cases}
	\]
}The final time is $t=0.038$. Reflective boundary conditions are imposed on the left and right boundaries, while periodic boundary conditions are used on the top and bottom boundaries. The Woodward--Colella blast wave problem contains strong shock interactions generated by large pressure jumps and is therefore a challenging test for the robustness of the scheme.
\begin{figure}[htbp]
	\vspace{-10pt}
	\centering
	\includegraphics[width=0.35\textwidth]{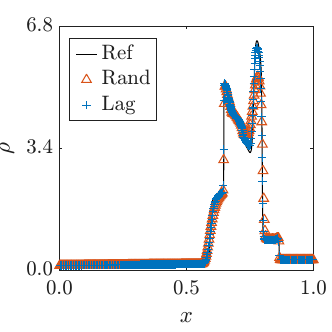}
	\includegraphics[width=0.35\textwidth]{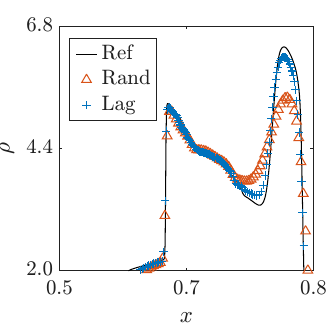}
	\caption{Density figures for the Woodward--Colella blast wave problem under random (Rand) and Lagrangian (Lag) mesh motions: full domain (left) and zoomed view (right).}
	\vspace{-20pt}
	\label{fig:Blast}
\end{figure}

Figure~\ref{fig:Blast} shows that the CE-ALE-DG scheme captures the main density structures under both mesh motions ($C_r=0.5$ and $\bm T_r=(0,0)^T$)  without visible spurious oscillations. In the zoomed view, the Lagrangian mesh velocity provides sharper resolution near the high-density region, and the result with the random mesh velocity remains free of noticeable spurious oscillations.
\subsubsection{Rotated Lax problem}
We then consider a rotated version of the Lax problem~\cite{LiShuQiu2021}. Define the rotated coordinates by $ s=\frac{x-y}{\sqrt{2}}, r=\frac{x+y}{\sqrt{2}}.$ The computational domain is $\Omega(0)=[-0.5,0.5]\times[-0.5,0.5]$, using a $100\times100$ structured quadrilateral mesh. The initial condition is given by {\setlength{\abovedisplayskip}{3pt}
	\setlength{\belowdisplayskip}{3pt}
	\[
	(\rho,\mu,\nu,p,\gamma)(x,y)=
	\begin{cases}
		(0.445, \frac{0.698}{\sqrt{2}}, -\frac{0.698}{\sqrt{2}}, 3.528,1.4), & s\leq0,\\
		(0.5, 0, 0, 0.571,1.4), & s>0.
	\end{cases}
	\]
}The final time is $t=0.16$. Exact boundary data obtained from the corresponding one-dimensional Riemann solution in the $s$-direction are imposed on all physical boundaries, with the velocity components transformed back to Cartesian coordinates.
\begin{figure}[htbp]
	\vspace{-20pt}
	\centering
	\captionsetup{skip=2pt}
	\includegraphics[width=0.35\textwidth]{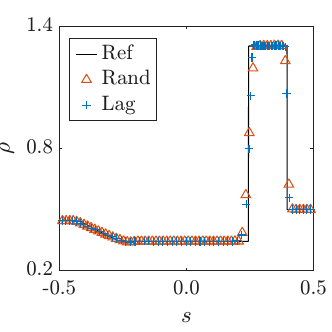}
	\includegraphics[width=0.35\textwidth]{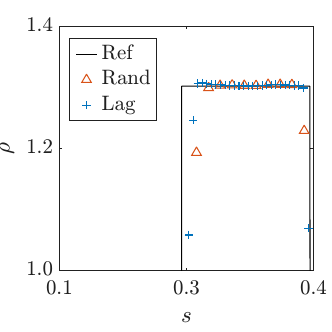}
	\caption{Density figures along $r=0$ for the rotated Lax problem under random (Rand) and Lagrangian (Lag) mesh motions: overall view (left) and zoomed view (right).}
	\label{fig:Lax}
	\vspace{-20pt}
\end{figure}

Figure~\ref{fig:Lax} shows that the CE-ALE-DG scheme captures the main wave structures of the rotated Lax problem under both random ($C_r=0.5$ and $\bm T_r=(0,0)^T$) and Lagrangian mesh motions. The KXRCF-OE procedure effectively controls spurious oscillations near the discontinuities. In the zoomed view, the Lagrangian mesh velocity leads to a mild concentration of mesh points near the discontinuity, which helps improve the local resolution.
\subsubsection{Sedov blast wave problem}
We next consider the Sedov blast wave problem~\cite{liska2003comparison}. The initial condition is prescribed as
{\setlength{\abovedisplayskip}{3pt}
	\setlength{\belowdisplayskip}{3pt}
	\[
	(\rho,\mu,\nu,E,\gamma)(x,y)=
	\begin{cases}
		(1,\ 0,\ 0,\ \dfrac{0.244816}{|\Omega_{\mathrm{ori}}(0)|},1.4),
		& (x,y)\in\Omega_{\mathrm{ori}}(0),\\[1.2ex]
		\left(1,\ 0,\ 0,\ 10^{-12},1.4\right),
		& \text{otherwise},
	\end{cases}
	\]}where $\Omega_{\mathrm{ori}}(0)$ denotes the lower-left cell adjacent to the origin, and $|\Omega_{\mathrm{ori}}(0)|$ is its area. The computation is performed on $\Omega(0)=[0,1.2]\times[0,1.2]$ up to
$t=1.0$ using a $40\times40$ structured quadrilateral mesh. Reflective boundary conditions are imposed on the left and bottom boundaries, while outflow conditions are imposed on the right and top boundaries. The random mesh motion uses $C_r=2.0$ and $\bm T_r=(0,0)^T$.
\begin{figure}[htbp]
	\centering
	\includegraphics[width=1.0\textwidth]{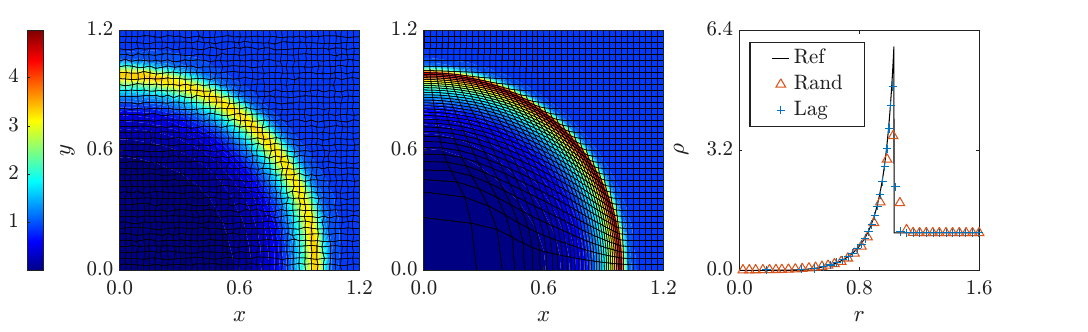}
	\caption{Density distributions for the Sedov blast wave problem under random (left) and Lagrangian (middle) mesh motions, and along $x=y$ (right).}
	\vspace{-20pt}
	\label{fig:Sedov_2D}
\end{figure}

Figure~\ref{fig:Sedov_2D} shows the results of the CE-ALE-DG scheme for this problem under the random and Lagrangian mesh velocities. The Lagrangian mesh motion resolves the high-density region more sharply, while the random-velocity result remains stable. The density along the diagonal $x=y$ also agrees well with the reference solution.
\subsubsection{Double Mach reflection problem}
We finally consider the double Mach reflection problem~\cite{woodward1984numerical}. The computational domain is $\Omega(0)=[0,4]\times[0,1],$ and the final time is $t=0.2$. The initial condition is a Mach 10 shock inclined at an angle of $60^\circ$ to the $x$-axis, given by
{
	\setlength{\abovedisplayskip}{3pt}
	\setlength{\belowdisplayskip}{3pt}
	\[
	(\rho,\mu,\nu,p)(x,y,0)=
	\begin{cases}
		\left(8,\ 8.25\cos\frac{\pi}{6},\ -8.25\sin\frac{\pi}{6},\ 116.5\right),
		& x\le \dfrac{1}{6}+\dfrac{y}{\sqrt{3}},\\[2ex]
		(1.4,\ 0,\ 0,\ 1),
		& x> \dfrac{1}{6}+\dfrac{y}{\sqrt{3}}.
	\end{cases}
	\]}
The left and right boundaries use post-shock and outflow conditions, respectively. On the bottom boundary, the post-shock state is imposed for $0\le x\le 1/6$ and a reflecting condition elsewhere. On the top boundary, the exact moving-shock condition with $x_s(y,t)=\frac{1}{6}+\frac{y+20t}{\sqrt{3}}$ assigns the post-shock state for $x\le x_s(y,t)$ and the pre-shock state otherwise. The domain is discretized using a $960\times240$ structured quadrilateral mesh. The random mesh motion is specified by
$C_r=0.5$ and $\bm T_r=(0,0)^T$.
\begin{figure}[htbp]
	\vspace{-10pt}
	\centering
	\includegraphics[width=0.8\textwidth]{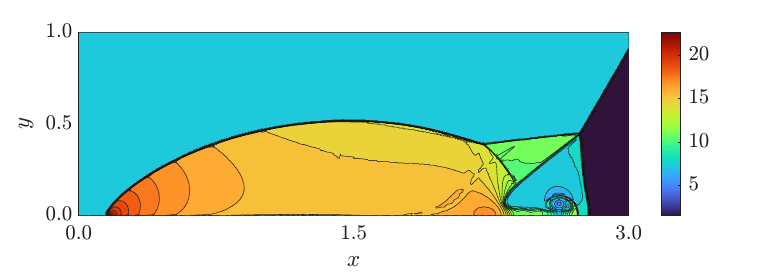}
	\caption{Thirty equally spaced density contours in $[1.5,22.7]$ for the double Mach reflection problem under random mesh motion, shown over $[0,3]\times[0,1]$.}
	\vspace{-20pt}
	\label{fig:DM}
\end{figure}

Figure~\ref{fig:DM} shows the density contours of the double Mach reflection problem under the prescribed random mesh motion. Despite the strong mesh motion, the CE-ALE-DG scheme captures the main shock structures and the small-scale features near the slip line without visible nonphysical oscillations.
\section{Conclusions}\label{sec:Conclusion}
This work addresses the loss of high-order accuracy in direct arbitrary Lagrangian--Eulerian discontinuous Galerkin (ALE-DG) methods when the mesh velocity does not admit a mesh-independent Lipschitz bound. The diagnostic results show that the geometric conservation law (GCL) alone is insufficient to prevent this order reduction. To the best of our knowledge, this is the first use of transport polynomial exactness (TPE($k$)) as a higher-order structural criterion for direct ALE-DG methods, extending exactness from constant states to transported polynomials.

Guided by TPE($k$), we derived a direct coefficient-evolution ALE-DG (CE-ALE-DG) formulation and proved that it satisfies the GCL and TPE($k$) without a mesh-velocity Lipschitz bound. The resulting coefficient update, however, is not conservative. We therefore designed a conservative correction that enforces local and global conservation for general solutions while vanishing for constant states and transported quadratic polynomials. This correction is the key to combining local conservation with the GCL and TPE($2$) in the third-order fully discrete scheme on both triangular and quadrilateral moving meshes.

Numerical experiments confirm the conservation, GCL, TPE($1$), and TPE($2$) properties of the fully discrete scheme on both triangular and quadrilateral meshes. The accuracy tests show that CE-ALE-DG retains third-order accuracy when the mesh-velocity Lipschitz estimate grows as $\mathcal O(h^{-1})$, whereas the reference direct ALE-DG scheme~\cite{fu2019simplex} exhibits order reduction under the same mesh motion, supporting the role of TPE($2$) in preserving high-order accuracy. Comparisons with and without the conservative correction confirm its importance, while strong-shock tests illustrate the robustness of the scheme under random and Lagrangian mesh motions. Future work will extend the framework to higher-order TPE($k$) properties, three-dimensional meshes, and adaptive mesh motion.

\end{document}